\documentclass{amsart}

\usepackage{amssymb,amsmath,esint}
\usepackage[all]{xy}
\usepackage{graphicx}
\usepackage{enumerate}

\usepackage{colortbl}
\usepackage{xcolor}

\usepackage{amsrefs}

\usepackage{verbatim}

\colorlet{linkequation}{blue}
\usepackage[colorlinks,plainpages=true,pdfpagelabels,hypertexnames=true,colorlinks=true,pdfstartview=FitV,linkcolor=blue,citecolor=red,urlcolor=black]{hyperref}

\PassOptionsToPackage{unicode}{hyperref}
\PassOptionsToPackage{naturalnames}{hyperref}

\definecolor{dgreen}{rgb}{0,0.5,0}
\definecolor{violet}{rgb}{0.5,0,0.5}
\definecolor{dred}{rgb}{0.7,0,0}
\definecolor{ddred}{rgb}{0.5,0,0}
\definecolor{dblue}{rgb}{0,0,0.5}
\definecolor{ddblue}{rgb}{0,0,0.3}

\newtheorem{theorem}{Theorem}[section]
\newtheorem{lemma}[theorem]{Lemma}
\newtheorem{corollary}[theorem]{Corollary}

\newtheorem{definition}[theorem]{Definition}
\newtheorem{proposition}[theorem]{Proposition}

\newtheorem{remark}[theorem]{Remark}

\numberwithin{equation}{section}

\DeclareMathOperator*{\ddiv}{div}

\newcommand{\capp}{\text{cap}_p}

\newcommand{\integral}[3]{\int_{#1} #2 \ #3}

\newcommand{\mint}[3]{\fint_{#1} #2 \ #3}

\newcommand{\norm}[1]{\left| #1\right|}
\newcommand{\Norm}[1]{\left|\hspace{-0.2mm}\left| #1 \right|\hspace{-0.2mm}\right|}
\newcommand{\gh}[1]{\left( #1\right)}
\newcommand{\mgh}[1]{\left\{ #1\right\}}
\newcommand{\bgh}[1]{\left[ #1\right]}
\newcommand{\vgh}[1]{\left< #1\right>}

\newcommand{\OO}{\Omega}

\newcommand{\ma}{\mathbf{a}}

\newcommand{\lpwwz}{L^p (0,T; W_0^{1,p}(\OO))}

\newcommand{\abs}[1]{\left| #1\right|}
\newcommand{\Abs}[1]{\left\| #1\right\|}

\newcommand{\iq}[1]{Q_{#1}^{\lambda}}

\newcommand{\pc}[1]{\lfloor {#1} \rfloor}

\newcommand{\la}{\lambda}
\newcommand{\La}{\Lambda}

\newcommand{\ka}{\kappa}

\newcommand{\M}{\mathcal{M}}

\newcommand{\bb}{\mathbb{R}}
\newcommand{\LL}{\mathcal{L}}

\newcommand{\loc}{loc}

\AtBeginDocument{%
   \def\MR#1{}
}

\def\XXint#1#2#3{{\setbox0=\hbox{$#1{#2#3}{\int}$}
    \vcenter{\hbox{$#2#3$}}\kern-.5\wd0}}

\begin{document}

\title[Singular parabolic $p$-Laplace type equations with measure data]{Marcinkiewicz regularity for singular parabolic $p$-Laplace type equations with measure data}

\author{Jung-Tae Park}
\address{J.-T. Park: School of Liberal Arts, Korea University of Technology and Education, Cheonan 31253, Republic of Korea}
\email{jungtae.park@koreatech.ac.kr}



\date{July 31, 2021 and, in revised form, April 9, 2022}


\begin{abstract}
We consider quasilinear parabolic equations with measurable coefficients when the right-hand side is a signed Radon measure with finite total mass, having $p$-Laplace type:
$$u_t - \textrm{div} \, \mathbf{a}(Du,x,t) = \mu \quad \textrm{in} \ \Omega \times (0,T) \subset \mathbb{R}^n \times \mathbb{R}.$$
In the singular range $\frac{2n}{n+1} <p \le 2-\frac{1}{n+1}$, we establish regularity estimates for the spatial gradient of solutions in the Marcinkiewicz spaces, under a suitable density condition of the right-hand side measure.
\end{abstract}

\maketitle



\section{Introduction}\label{Introduction}

We study some integrability results of the spatial gradient of solutions to nonlinear parabolic problems with measure data, having the $p$-Laplace type:
\begin{equation}\label{E:PME}
u_t - \ddiv \mathbf{a}(Du,x,t) = \mu \quad \text{in}  \  \OO_T,
\end{equation}
where $\mu$ is a signed Radon measure with finite total mass, that is, $|\mu|(\Omega_{T}) < \infty$. As usual the unknown is $u: \Omega_{T} \to \mathbb{R}$, $u=u(x,t)$, where $\Omega_{T} = \Omega \times (0, T)$ is a cylindrical domain with a bounded, open subset $\Omega \subset \mathbb{R}^n$, $n\geq 2$, and $T>0$. We write $Du:=D_x u$.
The vector field $\mathbf{a}=\mathbf{a}(\xi,x,t): \bb^n \times \bb^n \times \bb \rightarrow \bb^n$ is assumed to be measurable in $(x,t)$, continuous in $\xi$, and subject to the structure conditions
\begin{equation}\label{structure}
\begin{cases}
  \left|\mathbf{a}(\xi,x,t)\right| \leq \Lambda_1 |\xi|^{p-1},\\
  \left< \mathbf{a}(\xi_1,x,t)-\mathbf{a}(\xi_2,x,t), \xi_1-\xi_2 \right> \ge \La_0 \left( \norm{\xi_1}^2 + \norm{\xi_2}^2  \right)^{\frac{p-2}{2}} \norm{\xi_1-\xi_2}^2
\end{cases}
\end{equation}
for almost every $(x, t) \in \mathbb{R}^n \times \bb$, for  any $\xi_1, \xi_2, \xi \in \mathbb{R}^n$ and for some constants $\Lambda_1 \ge \Lambda_0 >0$. A typical prototype of \eqref{E:PME} is the parabolic $p$-Laplace equation
\begin{equation}\label{model eq}
u_t - \ddiv\gh{|Du|^{p-2}Du}  = \mu.
\end{equation}
According to the range of $p$, the equation \eqref{model eq} can be divided into three types. If $p=2$, then \eqref{model eq} just becomes the classical heat equation. At points where $|Du|=0$, the diffusivity coefficient $|Du|^{p-2}$ vanishes if $p>2$ and it blows up if $1<p<2$. Thus \eqref{model eq} for $p>2$ is called {\em degenerate parabolic $p$-Laplace equation}. When $1<p<2$, \eqref{model eq} is referred to as {\em singular parabolic $p$-Laplace equation}.

In this paper, we shall deal with the singular range
\begin{equation}\label{range of p}
  \frac{2n}{n+1} < p \le 2-\frac{1}{n+1},
\end{equation}
while the other singular range $2-\frac{1}{n+1} < p <2$ has been considered in \cite{Bar17}, and the degenerate range $p\ge2$ has been treated in \cite{Bar14,BD18}. The lower bound on $p$ in \eqref{range of p} is natural and sharp, since it is related to the existence of the (Barenblatt) fundamental solution (see for instance \cite[Chapter 11.4]{Vaz06}). Indeed, the fundamental solution $\Gamma$ is explicitly given for $\frac{2n}{n+1}<p<2$ by
\begin{equation*}\label{baren sol}
\Gamma(x,t) = t^{-n\alpha} \bgh{c(n,p) + \frac{2-p}{p} \alpha^{\frac{1}{p-1}} \gh{\frac{|x|}{t^\alpha}}^{\frac{p}{p-1}} }^{-\frac{p-1}{2-p}},
\end{equation*}
where $\alpha := \frac{1}{p(n+1) - 2n}$. The solution $\Gamma$ is well defined when $\alpha >0$; that is, $p > \frac{2n}{n+1}$.

Our aim of this paper is to establish Marcinkiewicz estimates for the spatial gradient of solutions to the singular parabolic measure data problems \eqref{E:PME} satisfying \eqref{structure} and \eqref{range of p}, under some decay conditions on the right-hand side measure (see Section~\ref{Main theorems}). To formulate our main results, we define certain so-called {\em Morrey-type condition} for a measure as follow. For a signed Radon measure $\mu$ with finite total mass, we say that $\mu$ satisfies a {\em Morrey-type condition}, written $\mu \in \LL^{1,\kappa}(\OO_T)$, provided that
\begin{equation*}
|\mu| (Q_{r}(z_0)) \le C_0 r^{N-\kappa} \quad (0 \le \kappa \le N, \ C_0 \ge 1)
\end{equation*}
holds for any standard parabolic cylinder $Q_{r}(z_0):= B_{r}(x_0) \times \gh{t_0 - r^{2}, t_0+r^{2}} \subset \OO_T$ of parabolic dimension $N:=n+2$.
Then we shall prove
\begin{equation*}\label{intro main reult}
  \mu \in \LL^{1,\kappa}(\OO_T) \ \ \text{for} \ \kappa_c < \kappa \le N \quad \Longrightarrow \quad Du \in \M_{\loc}^{\gamma}(\OO_T,\bb^n),
\end{equation*}
where two constants $\kappa_c = \kappa_c(n,\La_0,\La_1,p) \ge 1$ and $\gamma=\gamma(n,p,\kappa) \ge 1$ are determined explicitly later. Here $\M^{\gamma}(\OO_T,\bb^n)$ is the Marcinkiewicz space (see \eqref{Marcinkiewicz space} for definition).

To prove our main results, we use some covering arguments via a so-called {\em maximal function free technique} introduced in \cite{AM07} (see Section~\ref{la covering arguments}).
This approach is suitable to the situation in which it occurs the lack of homogeneity (roughly speaking it scales differently in time and space) of nonlinear parabolic problems, such as $p$-Laplacian with $p \neq 2$ or porous medium equation.
In this paper, this covering arguments are considered under the following intrinsic parabolic cylinders:
\begin{equation}\label{scaling}
  \frac{s}{r^{2}} = \lambda^{2-p} \ \ \text{with} \ \ \mint{Q_{r,s}(z_0)}{|Du|^{\theta}}{dxdt} \approx \lambda^{\theta} \ \ \text{for some} \ \theta \in (0,1),
\end{equation}
where $Q_{r,s}(z_0):=B_r(x_0) \times (t_0-s,t_0+s)$. We point out that the spatial gradient of a solution $u$ of \eqref{E:PME} may not belong to the $L^1$ space under \eqref{range of p}. For this reason, we need a notion of solution in a renormalized sense (see Definition~\ref{renormalsol}). Also, in this circumstance, we show a decay estimate  of the upper-level sets of $|Du|$ (see Section~\ref{Estimates on upper-level sets}), alongside with \eqref{scaling} and difference estimates (see Section~\ref{intrinsic comparison}).

There are several results concerning gradient regularity for parabolic $p$-Laplace type equations with measure data for $p>2-\frac{1}{n+1}$, as follows.
\begin{itemize}
  \item Potential estimates: e.g. \cite{DM11} for $p=2$, \cite{KM14b,KM14c} for $p\ge2$, and \cite{KM13b} for $2-\frac{1}{n+1} < p \le2$.
  \item Marcinkiewicz estimates: e.g. \cite{Bar14,BD18} for $p\ge2$, and \cite{Bar17} for $2-\frac{1}{n+1} < p <2$.
  \item Fractional differentiability: e.g. \cite{Bar14b,BH12,BCS21} for $p=2$.
  \item Calder\'{o}n-Zygmund type estimates: e.g. \cite{Ngu15,Ngu_pre,BP18b} for $p=2$, and \cite{BPS21} for $p>2-\frac{1}{n+1}$.
\end{itemize}

Compared to the results mentioned above, there are few regularity results for the case $\frac{2n}{n+1}< p \le 2-\frac{1}{n+1}$. We refer to \cite{PS22} for global Calder\'{o}n-Zygmund type estimates on nonsmooth domains. It is worthwhile to note that there are regularity estimates (see \cite{NP19,NP20,NP20b,DZ21}) for elliptic measure data problems with $1< p \le 2-\frac{1}{n}$.

The paper is organized as follows. In Section~\ref{preliminaries}, we introduce notation, terminologies and renormalized solutions, and we present our main results. In Section~\ref{intrinsic comparison}, we collect comparison estimates between our $p$-Laplace type problem and its reference problems. Finally, in Section~\ref{la covering arguments}, we prove our main results, by deriving decay estimates via covering arguments under intrinsic parabolic cylinders.

\section{Preliminaries and main results}\label{preliminaries}

    \subsection{Notation and definitions}\label{Notation and definitions}

We start with notations. Let us denote by $c$ a universal positive constant, which may change from line to line.  Let $B_{r} (x_0):=\mgh{ x \in \mathbb{R}^n : |x-x_0| < r }$ and $I_{r} (t_0) := \gh{t_0 - r^{2}, t_0+r^{2}}$. For $\lambda >0$, we denote by
\begin{equation*}\label{ipc}
Q_{r}^{\lambda} (z_0) := \underbrace{\mgh{ x \in \mathbb{R}^n : |x-x_0| < \lambda^{\frac{p-2}{2}}r }}_{=: B_{r}^{\lambda}(x_0)} \times I_{r}(t_0)
\end{equation*}
the {\em intrinsic parabolic cylinder} in $\bb^n \times \bb =: \bb^{n+1}$ with center $z_0:=(x_0,t_0) \in \bb^{n+1}$.
When $p=2$ or $\lambda=1$, $Q_{r}^{\lambda}(z_0) \equiv Q_{r}(z_0)$. Also we have $Q_{r}^{\lambda}(z_0) \subset Q_{r}(z_0)$ if $\lambda \geq 1$ and $p\le2$.
The concept of {\em intrinsic} means, roughly speaking, that the size of parabolic cylinders depends on the solution of a given PDE in some integral average sense; in particular, the formulation \eqref{scaling} can be rewritten as
\begin{equation*}\label{scaling2}
  \mint{Q_{r}^{\lambda}(z_0)}{|Du|^{\theta}}{dxdt} \approx \lambda^{\theta} \quad  \text{for some} \  \theta \in (0,1).
\end{equation*}
Under this setting, we shall consider various a priori estimates for the spatial gradient of a solution later.
We refer to \cite{DiB93,Urb08,KM14b} for further discussion about intrinsic scalings.

Let us define the truncation operator
\begin{equation}\label{T_k}
T_k(s) := \max\mgh{-k,\min\mgh{k,s}} \quad \text{for any} \ k>0 \ \text{and} \ s \in \bb.
\end{equation}
For each set $Q \subset \bb^{n+1}$, $|Q|$ is the $(n+1)$-dimensional Lebesgue measure of $Q$ and $\chi_Q$ is the usual characteristic function of $Q$.
For $f \in L_{loc}^1(\bb^{n+1})$, $\bar{f}_{Q}$ stands for the integral average of $f$ over a parabolic cylinder $Q \subset \bb^{n+1}$; that is,
$$\bar{f}_{Q} := \mint{Q}{f(z)}{dx dt} := \frac{1}{|Q|} \integral{Q}{f(z)}{dx dt}, \ \ \text{where } z:=(x,t).$$

Let us introduce a nonlinear parabolic capacity (see \cite{Pie83,DPP03,KKKP13,AKP15} for details), which is necessary to define our solution later.
For every open subset $Q \subset \OO_T$, the {\em $p$-parabolic capacity} of $Q$ is defined by
$$\capp(Q) := \inf\mgh{\|u\|_W : u \in W,   u \ge \chi_Q \ \text{a.e. in} \ \OO_T},$$
where $W := \mgh{u \in L^p (0,T; W_0^{1,p}(\OO)) : u_t \in L^{p'}(0,T; W^{-1,p'}(\OO))}$
endowed with the norm $\|u\|_W := \|u\|_{L^p (0,T; W_0^{1,p}(\OO))} + \|u_t\|_{L^{p'}(0,T; W^{-1,p'}(\OO))}.$
Here $p' := \frac{p}{p-1}$.

Let $\mathfrak{M}_b(\OO_T)$ (or $\mathfrak{M}_b(\OO)$, $\mathfrak{M}_b(0,T)$) be the space of all signed Radon measures on $\OO_T$ (or $\OO$, $(0,T)$, respectively) with finite total mass. Let $\mathfrak{M}_a(\OO_T)$ be the subspace of $\mathfrak{M}_b(\OO_T)$ of the measures that are absolutely continuous with respect to the $p$-parabolic capacity, let $\mathfrak{M}_s(\OO_T)$ be the space of finite signed Radon measures in $\OO_T$ with support on a set of zero $p$-parabolic capacity, and let $C_b(\OO_T)$ be the space of all bounded and continuous functions on $\OO_T$.  A measure $\mu \in \mathfrak{M}_b(\OO_T)$ can be written as a sum of two measures as follows: $\mu = \mu_a +\mu_s$, where $\mu_a \in \mathfrak{M}_a(\OO_T)$ and $\mu_s \in \mathfrak{M}_s(\OO_T)$, see \cite[Lemma 2.1]{FST91}. Also, $\mu_a \in \mathfrak{M}_a(\OO_T)$ if and only if $\mu_a = f + g_t + \ddiv G,$
where $f \in L^1(\OO_T)$, $g \in L^p (0,T; W_0^{1,p}(\OO))$ and $G \in L^{p'}(\OO_T)$, see \cite{DPP03, KR19}.
We denote by $\mu^+$ and $\mu^-$ the positive and negative parts of a measure $\mu \in \mathfrak{M}_b(\OO_T)$, respectively. We write $|\mu| := \mu^+ + \mu^-$. We say that a sequence $\{\mu_k\} \subset \mathfrak{M}_b(\OO_T)$ converges {\em tightly} (or {\em in the narrow topology of measures}) to $\mu \in \mathfrak{M}_b(\OO_T)$ if
$$\lim_{k\to\infty} \integral{\OO_T}{\varphi}{d\mu_k} =\integral{\OO_T}{\varphi}{d\mu}  \quad \text{for every} \ \varphi \in C_b(\OO_T).$$

Finally, we define a certain function space. For $0<\gamma<\infty$, the space $\M^{\gamma}(\OO_T,\bb^l)$ is the so-called {\em Marcinkiewicz space} (or the {\em weak-$L^{\gamma}$ space}), defined as the set of all measurable maps $f: \OO_T \to \bb^l$ such that
\begin{equation}\label{Marcinkiewicz space}
\|f\|_{\mathcal{M}^{\gamma}(\OO_T,\bb^l)} := \sup_{\la>0} \lambda \left|\{z\in \Omega_{T}: |f(z)| > \lambda \}\right|^{\frac{1}{\gamma}} < \infty.
\end{equation}
We observe the following connection between the Marcinkiewicz and Lebesgue spaces: $L^{\gamma}(\OO_T,\bb^l) \subset \M^{\gamma}(\OO_T,\bb^l) \subset L^{\gamma-\varepsilon}(\OO_T,\bb^l)$ for any $\varepsilon \in (0,\gamma)$.
We refer to \cite[Chapter 1]{Gra14} for various properties for the Marcinkiewicz space.


\subsection{Main results}\label{Main theorems}

We start by introducing a suitable notion of a solution. Our solution $u$ will be treated in a very weak sense because our solution does not generally belong to the usual energy space. Moreover, under \eqref{range of p}, the spatial gradient of a solution may not be in $L^1(\OO_T)$ (see \cite[Section 1.3]{KM13b}). To overcome this situation, we introduce the following: if $u$ is a measurable function defined in $\OO_T$ such that $u$ is finite almost everywhere and $T_k(u) \in \lpwwz$ for any $k>0$, then there exists a unique measurable function $U$ such that $DT_k(u)=U\chi_{\{|u|<k\}}$ a.e. in $\OO_T$ for all $k>0$.  We define the spatial gradient of $u$ as the function $U$ and denote $Du:=U$. 
Now we define the following notion of a solution.
\begin{definition}[See \cite{PP15}]\label{renormalsol}
Let $\mu = \mu_a + \mu_s \in \mathfrak{M}_b(\OO_T)$, where $\mu_a \in \mathfrak{M}_a(\OO_T)$ and $\mu_s \in \mathfrak{M}_s(\OO_T)$. A function $u \in L^1(\OO_T)$ is a {\em renormalized solution} of the Cauchy-Dirichlet problem
\begin{equation}
\label{pem1}\left\{
\begin{alignedat}{3}
u_t -\ddiv \mathbf{a}(Du,x,t)  &= \mu &&\quad \text{in}  \  \OO_T, \\
u  &= 0 &&\quad \text{on} \  \partial_p \OO_T,
\end{alignedat}\right.
\end{equation}
satisfying \eqref{structure} and \eqref{range of p} if $T_k(u) \in \lpwwz$ for every $k>0$ and the following property holds: for any $k>0$ there exist two sequences $\{\nu_k^+\}$, $\{\nu_k^-\}$ of nonnegative measures in $\mathfrak{M}_a(\OO_T)$ such that
\begin{equation*}
\nu_k^+ \to \mu_s^+,  \ \nu_k^- \to \mu_s^- \quad \text{tightly as} \ k\to \infty
\end{equation*}
and
\begin{equation}\label{tweaksol}
-\integral{\OO_T}{T_k(u)\varphi_t}{dxdt} + \integral{\OO_T}{\vgh{\mathbf{a}(DT_k(u),x,t), D\varphi}}{dxdt}
= \integral{\OO_T}{\varphi}{d\mu_k}
\end{equation}
for every $\varphi \in W \cap L^{\infty}(\OO_T)$ with $\varphi(\cdot,T)=0$, where $\mu_k := \mu_a + \nu_k^+ - \nu_k^-$.
\end{definition}
Here the parabolic boundary of $\OO_T$ is $\partial_p \OO_T:= \gh{\partial \OO \times [0,T]} \cup \gh{\OO \times \{0\}}$. We refer to \cite[Section 1.1]{PS22} and the references given there for further discussion of renormalized solutions.

Next, we define a density condition of a measure.
\begin{definition}\label{mudecay}
\begin{enumerate}[(i)]
  \item\label{mu} For $\mu \in \mathfrak{M}_b(\OO_T)$, we say that $\mu$ satisfies a {\em Morrey-type condition}, written $\mu \in \LL^{1,\kappa}(\OO_T)$, provided that
      \begin{equation*}
      |\mu| (Q_{r}(z_0)) \le C_0 r^{N-\kappa} \quad (0 \le \kappa \le N:= n+2, \ C_0 \ge 1)
      \end{equation*}
      holds for any standard parabolic cylinder $Q_{r}(z_0)\subset\OO_T$.

  \item\label{mu1} Similarly, for $\mu_1 \in \mathfrak{M}_b(\OO)$, we define
    \begin{equation*}
    \mu_1 \in \LL^{1,\kappa_1}(\OO) \iff |\mu_1|(B_r(x_0)) \le C_1 r^{n- \kappa_1} \quad (0 \le \kappa_1 \le n, \ C_1 \ge 1)
    \end{equation*}
    holds for any ball $B_{r}(x_0)\subset\OO$.

  \item\label{mu2} Also, for $\mu_2 \in \mathfrak{M}_b(0,T)$, we define
  \begin{equation*}
  \mu_2 \in \LL^{1,\kappa_2}(0,T) \iff |\mu_2|(I_{r}(t_0)) \le C_2 r^{2-\kappa_2} \quad (0 \le \kappa_2 \le 2, \ C_2 \ge 1)
  \end{equation*}
  holds for any interval $I_{r}(t_0)\subset (0,T)$.
\end{enumerate}
\end{definition}

Note that $\LL^{1,N}(\OO_T) \equiv \mathfrak{M}_b(\OO_T)$. For example, the Dirac measure charging a point in $\OO_T$ belongs to $\LL^{1,N}(\OO_T)$.

We are ready to state the first main result of this paper.
\begin{theorem}\label{main thm1}
  Let $\frac{2n}{n+1}<p\le 2-\frac{1}{n+1}$. There is a constant $\kappa_c = \kappa_c(n,\La_0,\La_1,p) \ge 1$ such that if $u$ is a renormalized solution of the problem \eqref{pem1} under \eqref{structure} with $\mu \in \LL^{1,\kappa}(\OO_T)$ for $\kappa_c < \kappa \le N$, then
  \begin{equation}\label{main thm1-r1}
    Du \in \M_{\loc}^{\gamma}(\OO_T,\bb^n), \ \ \text{where} \ \gamma := \frac{\kappa}{\kappa-1}\max\mgh{p-1, \ \frac{1}{2}\gh{p-\frac{n(2-p)}{\kappa}}}.
  \end{equation}
  Moreover, for any given $\theta$ satisfying
  \begin{equation}\label{ka-range}
  \max\mgh{\frac{n+2}{2(n+1)},\frac{n(2-p)}{2}} < \theta < p-\frac{n}{n+1} \le 1,
  \end{equation}
  there is a constant $c=c(n,\La_0,\La_1,p,\kappa,C_0,\theta) \ge 1$ such that
  \begin{equation}\label{main thm1-r2}
  \Abs{Du}_{\M^{\gamma}(Q_{R},\bb^n)}^\gamma \le cR^N \mgh{ \bgh{\frac{|\mu|(Q_{2R})}{|Q_{2R}|}}^{d} + \gh{\mint{Q_{2R}}{|Du|^{\theta}}{dxdt}}^{\frac{d\gamma}{\theta}} +1}
  \end{equation}
  for any standard parabolic cylinder $Q_{2R} \equiv Q_{2R}(z_0) \Subset \OO_T$, where the scaling deficit $d$ is defined by
  \begin{equation}\label{scaling deficit}
    d := \frac{2\theta}{2\theta-n(2-p)}.
  \end{equation}
\end{theorem}

Theorem~\ref{main thm1} provides a precise quantitative estimate of the spatial gradient of a renormalized solution in terms of the Marcinkiewicz space, under the assumption that the measure on the right-hand side satisfies the Morrey-type condition. Roughly speaking, the less concentrated the measure $\mu$ is (i.e. the smaller $\kappa$ is), the better the integrability of $Du$ is (i.e. the bigger $\gamma$ is).

\begin{remark}\label{main rmk1}
\begin{enumerate}[(i)]
\item Note that the value of $\gamma$ in \eqref{main thm1-r1} is
\begin{equation*}
\gamma = \left\{
\begin{alignedat}{3}
&\frac{\kappa(p-1)}{\kappa-1} &&\quad \text{if } \ 1<\kappa\le n,\\
&\frac{\kappa}{2(\kappa-1)}\gh{p-\frac{n(2-p)}{\kappa}} &&\quad \text{if }\ n \le \kappa \le N.
\end{alignedat}\right.
\end{equation*}

\item\label{R:m2} When $\kappa=N$, the value $\gamma$ has the minimum $p-\frac{n}{n+1}$. Also, $\gamma \nearrow p(1+\sigma)$ when $\kappa \searrow \kappa_c$, see Section~\ref{Marcinkiewicz estimates} for details.  Here $\sigma$ is the constant coming from a higher integrability for homogeneous problems (Lemma~\ref{high int}). Thus, we have
    $$p-\frac{n}{n+1} \le \gamma < p(1+\sigma) \quad \text{for } \ \kappa_c < \kappa \le N.$$

\item As $p\searrow \frac{2n}{n+1}$, the constant $c$ in \eqref{main thm1-r2} blows up.

\item The ranges of both exponents $p$ and $\theta$ in Theorem~\ref{main thm1} come from combining Lemmas~\ref{compa1} and \ref{high int} below. Also, the exponent $\theta$ is not empty since $p>\frac{2n}{n+1}$; see Remark~\ref{comparison remark} for details.

\item The scaling deficit like \eqref{scaling deficit} occurs when we study regularity theories for PDE having anisotropic structures such as parabolic $p$-Laplace ($p \neq 2$) equations, see for example  \cite{KL00,AM07,KM13b,KM14b,KM14c,Bar14,Bar17,BD18,BPS21,PS22,Bog07,BOR13}.
\end{enumerate}
\end{remark}

Theorem~\ref{main thm1} also gives the following direct consequence when $\mu$ is merely a finite signed Radon measure (i.e. $\mu \in \LL^{1,N}(\OO_T) \equiv \mathfrak{M}_b(\OO_T)$). This result was found in \cite{AMST99} without a local estimate.
\begin{corollary}\label{main thm2}
  Let $\frac{2n}{n+1}<p\le 2-\frac{1}{n+1}$. If $u$ is a renormalized solution of \eqref{pem1} under \eqref{structure} satisfying $\mu \in \mathfrak{M}_b(\OO_T)$, then $Du \in \M_{\loc}^{p-\frac{n}{n+1}}(\OO_T,\bb^n)$. Moreover, the estimate \eqref{main thm1-r2} holds for $\kappa=N$ and $\gamma=p-\frac{n}{n+1}$.
\end{corollary}

\begin{remark}
In the case $\mu \in \mathfrak{M}_b(\OO_T)$, there are many results regarding the existence of a solution, such that
\begin{equation*}
  Du \in L_{loc}^q(\OO_T,\bb^n) \quad \text{for all } \ q < p-\frac{n}{n+1},
\end{equation*}
see for instance \cite{BG89,BDGO97,Pri97,Pet08,PPP11}. We also refer to \cite{Min11b,PS22} for further discussions in the literature. Thus Corollary~\ref{main thm2} provides a sharp integrability of $Du$.
\end{remark}

Next, if the measure $\mu$ can be decomposed into space and time components (see \eqref{decom} below), then we obtain the following Marcinkiewicz bound.
\begin{theorem}\label{main thm3}
Let $\frac{2n}{n+1}<p\le 2-\frac{1}{n+1}$ and let $u$ be a renormalized solution of the problem \eqref{pem1} under \eqref{structure}. Suppose that the following decomposition holds:
\begin{equation}\label{decom}
    \mu = \mu_1 \otimes \mu_2,
\end{equation}
where $\mu_1 \in L^{\infty}(\OO)$ and $\mu_2 \in \LL^{1,\kappa_2}(0,T)$ for some $\kappa_2 \in (1,2]$. Then there exists a constant $\kappa_{2,c} = \kappa_{2,c}(n,\La_0,\La_1,p) \ge 1$ such that  for $\kappa_{2,c} <\ka_2 \le2$, we have
\begin{equation}\label{def:gamma2}
Du \in \M_{\loc}^{\gamma_2}(\OO_T,\bb^n), \ \ \text{where} \ \gamma_2 := \frac{p\kappa_2}{2(\kappa_2-1)}.
\end{equation}
Moreover, for any $Q_{2R}(z_0) \Subset \OO_T$, a local estimate similar to \eqref{main thm1-r2} holds with $\gamma_2$ replacing $\gamma$.
\end{theorem}

\begin{remark}\label{R:mm}
\begin{enumerate}[(i)]
  \item 
      By analogy with Remark~\ref{main rmk1}\,\eqref{R:m2}, we see that $p\le\gamma_2 < p(1+\sigma)$ for $\kappa_{2,c} < \kappa_2 \le 2$.

  \item Note that $\gamma_2\ge \gamma$ for $1<\kappa=\kappa_2 \le2$. Under \eqref{decom}, \eqref{def:gamma2} gives a higher integrability compared to \eqref{main thm1-r1}.

  \item In the case $\mu_1 \in \LL^{1,\kappa_1}(\OO)$ for $\kappa_1 \in (1,n]$ and $\mu_2 \in L^{\infty}(0,T)$, we can obtain an integrability of $Du$ like \eqref{def:gamma2}. However, in our approach, this integrability is not improved compared to Theorem~\ref{main thm1}; see \cite[Section 1]{Bar17} for detailed explanations.
\end{enumerate}
\end{remark}

Theorem~\ref{main thm1} can be improved under more regular vector field than \eqref{structure}. We consider the vector field $\mathbf{a}=\mathbf{a}(\xi,x,t)$ measurable in $(x,t)$ and $C^1$-regular in $\xi$, satisfying
\begin{equation}\label{str1}
\begin{cases}
|\mathbf{a}(\xi,x,t)| + |\xi||D_{\xi}\mathbf{a}(\xi,x,t)| \le \Lambda_1 |\xi|^{p-1},\\
\Lambda_0 |\xi|^{p-2}|\eta|^2 \le \left< D_{\xi}\mathbf{a}(\xi,x,t)\eta,\eta \right>
\end{cases}
\end{equation}
for a.e. $(x, t) \in \mathbb{R}^n \times \bb$, for every $\eta \in \mathbb{R}^n$, $\xi \in \bb^n
\setminus \{0\}$ and for some $\Lambda_1 \ge \Lambda_0 >0$.
Note that \eqref{str1} implies the monotonicity condition $\eqref{structure}_2$. For an improvement of Theorem~\ref{main thm1}, we also need the following condition. Let $\delta,R_0>0$. We say that the vector field $\mathbf{a}(\xi,x,t)$ is {\em $(\delta,R_0)$-BMO} if
\begin{equation}\label{2-small BMO2}
\sup_{t_1,t_2 \in \bb} \sup_{0<r \le R_0} \sup_{y\in\bb^n} \fint_{t_1}^{t_2} \mint{B_r(y)}{\Theta\gh{\mathbf{a},B_r(y)}(x,t)}{dx dt} \le \delta,
\end{equation}
where
\begin{equation*}\label{2-AA}
\Theta\gh{\mathbf{a},B_r(y)}(x,t) := \sup_{\xi \in \bb^n \setminus \{0\}} \left|\frac{\mathbf{a}(\xi,x,t)}{|\xi|^{p-1}} - \mint{B_r(y)}{\frac{\mathbf{a}(\xi,\tilde{x},t)}{|\xi|^{p-1}}}{d\tilde{x}}  \right|.
\end{equation*}
We remark that the map $x \mapsto \frac{\mathbf{a}(\xi,x,t)}{|\xi|^{p-1}}$ is of BMO (Bounded Mean Oscillation) such that its BMO seminorm is less than $\delta$, uniformly in $\xi$ and $t$. This condition allows merely measurability in $t$-variable and discontinuity in $x$-variable. It also includes VMO (Vanishing Mean Oscillation) condition.

Finally, we obtain the following regularity result.
\begin{theorem}\label{main thm4}
  Let $\frac{2n}{n+1}<p\le 2-\frac{1}{n+1}$. Assume that the vector field $\mathbf{a}$ satisfies \eqref{str1} and a $(\delta,R_0)$-BMO condition for some $\delta, R_0>0$. If $u$ be a renormalized solution of the problem \eqref{pem1} with $\mu \in \LL^{1,\kappa}(\OO_T)$ for $1 < \kappa \le N$, then we have $Du \in \M_{\loc}^{\gamma}(\OO_T,\bb^n)$, where $\gamma$ is given by \eqref{main thm1-r1} with the range $p-\frac{n}{n+1} \le \gamma < \infty$. Moreover the estimate \eqref{main thm1-r2} holds.
\end{theorem}

We emphasize that all the main results in this section are consistent with the results in \cite{Bar17} for the singular case $2-\frac{1}{n+1} < p <2$. Specifically, the exponents $\gamma$ and $\gamma_2$ in \eqref{main thm1-r1} and \eqref{def:gamma2} are precisely the same as those in \cite{Bar17}, and so is the Marcinkiewicz estimate \eqref{main thm1-r2} with $\theta=1$. Thus, our results naturally extend the results in \cite{Bar17} to another singular case $\frac{2n}{n+1}<p\le 2-\frac{1}{n+1}$. We refer to \cite{Bar14,BD18} for the degenerate case $p\ge2$. See also \cite{Min07,Min10} for counterparts of elliptic problems.


\section{Comparison estimates}\label{intrinsic comparison}

In this section we derive comparison estimates between the problem \eqref{pem1} and its references problems, under the assumptions on the vector field $\mathbf{a}(\xi,x,t)$ and the measure $\mu$. (see Propositions~\ref{comparison1} and \ref{comparison2}).
From Definition~\ref{renormalsol}, we may regard $T_k(u) \in \lpwwz$ as a weak solution of \eqref{tweaksol} with $\mu_k \in L^{p'}(0,T; W^{-1,p'}(\OO))$. Throughout this section, we replace $T_k(u)$ by $u$ and $\mu_k$ by $\mu$.

Let $w$ be the unique weak solution to the Cauchy-Dirichlet problem
\begin{equation}
\label{pem2}\left\{
\begin{alignedat}{3}
w_t -\ddiv \mathbf{a}(Dw,x,t) &= 0 &&\quad \text{in} \ \iq{4r}(z_0) \Subset \OO_T, \\
w &= u &&\quad\text{on} \ \partial_p \iq{4r}(z_0),
\end{alignedat}\right.
\end{equation}
where the vector field $\mathbf{a}$ satisfies \eqref{structure}.
In this section, we for simplicity omit the center $z_0$ in $Q_{4r}^\la(z_0)$.

We first give a comparison estimate for $Du-Dw$ as follows:
\begin{lemma}[See {\cite[Lemma 3.1]{PS22}}]\label{compa1}
Let $\frac{3n+2}{2n+2}<p\le 2-\frac{1}{n+1}$, let $u$ be a weak solution of \eqref{tweaksol} and let $w$ as in \eqref{pem2} under \eqref{structure}.
Then there exists a constant $ c=c(n,\La_0,p,\theta)\ge1$ such that
\begin{equation*}\label{compa1-r}
\begin{aligned}
\gh{\mint{\iq{4r}}{|Du-Dw|^{\theta}}{dxdt}}^{\frac{1}{\theta}} &\le c \bgh{ \frac{|\mu|(\iq{4r})}{|\iq{4r}|^{\frac{n+1}{n+2}}} }^{\frac{n+2}{p(n+1)-n}}\\
&\quad + c \bgh{ \frac{|\mu|(\iq{4r})}{|\iq{4r}|^{\frac{n+1}{n+2}}} } \gh{\mint{\iq{4r}}{|Du|^{\theta}}{dxdt}}^{\frac{(2-p)(n+1)}{\theta(n+2)}}
\end{aligned}
\end{equation*}
for any constant $\theta$ such that $\frac{n+2}{2(n+1)} < \theta < p - \frac{n}{n+1} \le 1$.
\end{lemma}

We point out that Lemma~\ref{compa1} also holds for $p>2-\frac{1}{n+1}$; see \cite[Lemma 4.3]{KM13b} for $2-\frac{1}{n+1} < p \le 2$, and \cite[Lemma 4.1]{KM14b} for $p \ge 2$.

We next introduce a higher integrability result for $Dw$.
\begin{lemma}[See {\cite{KL00,PS22}}]\label{high int}
Let $\frac{2n}{n+2}<p\le2$ and let $\frac{n(2-p)}{2}<\theta\le p$. If $w$ is the weak solution of \eqref{pem2} under \eqref{structure} satisfying
\begin{equation*}\label{hi-c}
\mint{\iq{4r}}{|Dw|^{\theta}}{dxdt} \le c_w \lambda^{\theta}
\end{equation*}
for some constant $c_w\ge1$, then there exist two constants $\sigma=\sigma(n,\La_0,\La_1,p,\theta)>0$ and $c=c(n,\La_0,\La_1,p,\theta,c_w)\ge1$ such that
\begin{equation*}\label{hi-r}
\mint{\iq{2r}}{|Dw|^{p(1+\sigma)}}{dx dt} \le c\lambda^{p(1+\sigma)}.
\end{equation*}
\end{lemma}

We remark that Lemma~\ref{high int} also holds for $p\ge2$, see \cite{KL00,BPS21}.

If the measure $\mu$ satisfies a Morrey-type condition (see Definition~\ref{mudecay}), then we have the following relation:
\begin{lemma}\label{compa2}
Let $\lambda \ge 1$ and let $\mu \in \mathcal{L}^{1,\kappa}(\Omega_{T})$ for some $1 < \kappa \leq N:=n+2$. Assume that $\gamma$ is given by \eqref{main thm1-r1}. If the relation
\begin{equation}\label{compa2-r1}
\left[\frac{|\mu|(Q_{4r}^{\lambda})}{|Q_{4r}^{\lambda}|}\right]^{\frac{1}{\gamma}} \le \delta \lambda
\end{equation}
holds for some constant $\delta \in (0,1)$, then we have
\begin{equation*}
\left[\frac{|\mu|(Q_{4r}^{\lambda})}{|Q_{4r}^{\lambda}|^{\frac{n+1}{n+2}}}\right]^{\frac{n+2}{p(n+1)-n}}	
\leq c \delta^{\frac{\gamma(\kappa-1)(n+2)}{\kappa(p(n+1)-n)}} \lambda
\end{equation*}
for some constant $c=c(n,p, \kappa, C_0) \geq 1$, where $C_0$ is given by Definition~\ref{mudecay}\,\eqref{mu}.
\end{lemma}

\begin{proof}
For simplicity, we write $\beta:=p(n+1)-n$. We compute
\begin{equation}\label{compa2_00}
\begin{aligned}
	\left[ \frac{|\mu| (Q_{4r}^{\lambda})}{|Q_{4r}^{\lambda}|^{\frac{n+1}{n+2}}}\right]^{\frac{n+2}{\beta}}
	&= \left[ \frac{|\mu| (Q_{4r}^{\lambda})}{|Q_{4r}^{\lambda}|}\right]^{\frac{n+2}{\beta}} |Q_{4r}^{\lambda}|^{\frac{1}{\beta}} \\
	&= \left[ \frac{|\mu| (Q_{4r}^{\lambda})}{|Q_{4r}^{\lambda}|}\right]^{\alpha \frac{n+2}{\beta}}
	\left[ \frac{|\mu| (Q_{4r}^{\lambda})}{|Q_{4r}^{\lambda}|}\right]^{(1-\alpha)\frac{n+2}{\beta}}|Q_{4r}^{\lambda}|^{\frac{1}{\beta}},
\end{aligned}
\end{equation}
where $\alpha \in (0,1)$ is to be determined later.

First, we assume that $1<\kappa\le n$. The intrinsic parabolic cylinder $Q_{4r}^{\lambda}$ can be covered by finitely many (at most $2\lfloor\lambda^{2-p}\rfloor$) standard parabolic cylinders with radius $\lambda^{\frac{p-2}{2}}4r$. Combining this property and Definition~\ref{mudecay}\,\eqref{mu}, we deduce
\begin{equation}\label{compa2_01.1}
\frac{|\mu| (Q_{4r}^{\lambda})}{|Q_{4r}^{\lambda}|} \le c \frac{\lambda^{2-p}(\lambda^{\frac{p-2}{2}}4r)^{N-\kappa}}{\lambda^{\frac{n(p-2)}{2}}(4r)^N} \le c \lambda^{\frac{\kappa(2-p)}{2}}r^{-\kappa}.
\end{equation}
Inserting \eqref{compa2-r1} and \eqref{compa2_01.1} into the right-hand side of \eqref{compa2_00} yields
\begin{equation}\label{compa2_02.1}
\begin{aligned}
	\left[ \frac{|\mu| (Q_{4r}^{\lambda})}{|Q_{4r}^{\lambda}|^{\frac{n+1}{n+2}}}\right]^{\frac{n+2}{\beta}}
	&\leq c (\delta\lambda)^{\frac{\gamma\alpha(n+2)}{\beta}} \gh{\lambda^{\frac{\kappa(2-p)}{2}}r^{-\kappa}}^{\frac{(1-\alpha)(n+2)}{\beta}} \gh{\lambda^{\frac{n(p-2)}{2}}r^{n+2}}^{\frac{1}{\beta}}\\
&\le c \delta^{\frac{\gamma\alpha(n+2)}{\beta}} \lambda^{\frac{\gamma\alpha(n+2)}{\beta}+\frac{(2-p)(\kappa(1-\alpha)(n+2)-n)}{2\beta}} r^{\tilde{\alpha}}
\end{aligned}
\end{equation}
for some constant $c=c(n,p, \kappa, C_0)\ge1$, where $\tilde{\alpha}:=\frac{(n+2)(1-\kappa(1-\alpha))}{\beta}$.

On the other hand, we assume that $n\le \kappa \le N$. Definition~\ref{mudecay}\,\eqref{mu} provides
\begin{equation}\label{compa2_01}
\frac{|\mu| (Q_{4r}^{\lambda})}{|Q_{4r}^{\lambda}|} \le \frac{|\mu| (Q_{4r})}{|Q_{4r}^{\lambda}|} \le c \frac{(4r)^{N-\kappa}}{\lambda^{\frac{n(p-2)}{2}}(4r)^N} \le c \lambda^{\frac{n(2-p)}{2}}r^{-\kappa}.
\end{equation}
Similar to \eqref{compa2_02.1}, we insert \eqref{compa2-r1} and \eqref{compa2_01} into the right-hand side of \eqref{compa2_00}, to discover
\begin{equation}\label{compa2_02}
\left[ \frac{|\mu| (Q_{4r}^{\lambda})}{|Q_{4r}^{\lambda}|^{\frac{n+1}{n+2}}}\right]^{\frac{n+2}{\beta}} \le c \delta^{\frac{\gamma\alpha(n+2)}{\beta}} \lambda^{\frac{\gamma\alpha(n+2)}{\beta}+\frac{n(2-p)((1-\alpha)(n+2)-1)}{2\beta}} r^{\tilde{\alpha}}.
\end{equation}

Now we fix $\alpha \in (0,1)$ such that $\tilde{\alpha} = 0$ ; that is, $\alpha = \frac{\kappa - 1}{\kappa}$.
From such a choice of $\alpha$ and the definition of $\gamma$, both \eqref{compa2_02.1} and \eqref{compa2_02} imply
\begin{equation*}
\left[ \frac{|\mu| (Q_{4r}^{\lambda})}{|Q_{4r}^{\lambda}|^{\frac{n+1}{n+2}}}\right]^{\frac{n+2}{\beta}}
	\leq c \delta^{\frac{\gamma(\kappa-1)(n+2)}{\kappa\beta}} \lambda,
\end{equation*}
which completes the proof.
\end{proof}

If the measure $\mu$ admits a favorable decomposition, we obtain
\begin{lemma}\label{compa3}
Let $\lambda \ge 1$. Assume that the measure $\mu$ has the following decomposition
\begin{equation*}
  \mu = \mu_1 \otimes \mu_2,
\end{equation*}
where $\mu_1 \in L^{\infty}(\OO)$ and $\mu_2 \in \LL^{1,\kappa_2}(0,T)$ for some $\kappa_2 \in (1,2]$. Also, assume that $\gamma_2$ is given by \eqref{def:gamma2}. If the relation
\begin{equation}\label{compa3-r3}
\left[\frac{|\mu|(Q_{4r}^{\lambda})}{|Q_{4r}^{\lambda}|}\right]^{\frac{1}{\gamma_2}} \le \delta \lambda
\end{equation}
holds for some constant $\delta \in (0,1)$, then we have
\begin{equation}\label{compa3-r4}
\left[\frac{|\mu|(Q_{4r}^{\lambda})}{|Q_{4r}^{\lambda}|^{\frac{n+1}{n+2}}}\right]^{\frac{n+2}{p(n+1)-n}}	
\leq c \delta^{\frac{\gamma_2(\kappa_2-1)(n+2)}{\kappa_2(p(n+1)-n)}} \lambda
\end{equation}
for some constant $c=c(n,p, \kappa, C_2, \|\mu_1\|_{L^{\infty}(\OO)}) \geq 1$, where $C_2$ is given by Definition~\ref{mudecay}\,\eqref{mu2}.
\end{lemma}

\begin{proof}
For simplicity, we write $\beta:=p(n+1)-n$. Our assumption implies
\begin{equation}\label{decom2-1}
\frac{|\mu_1|(B_{4r}^{\lambda})}{|B_{4r}^{\lambda}|} \leq \|\mu_1\|_{L^{\infty}(\OO)}
\quad \text{and} \quad
\frac{|\mu_2| (I_{4r})}{|I_{4r}|} \leq c r^{-\kappa_2}.
\end{equation}
Inserting \eqref{compa3-r3} and \eqref{decom2-1} into the right-hand side of \eqref{compa2_00} yields
\begin{equation*}
\begin{aligned}
	\left[ \frac{|\mu| (Q_{4r}^{\lambda})}{|Q_{4r}^{\lambda}|^{\frac{n+1}{n+2}}}\right]^{\frac{n+2}{\beta}}
	&\leq c (\delta\lambda)^{\frac{\gamma_2\alpha(n+2)}{\beta}} r^{\frac{-\kappa_2(1-\alpha)(n+2)}{\beta}} \gh{\lambda^{\frac{n(p-2)}{2}}r^{n+2}}^{\frac{1}{\beta}}\\
&\le c \delta^{\frac{\gamma_2\alpha(n+2)}{\beta}} \lambda^{\frac{\gamma_2\alpha(n+2)}{\beta}+\frac{n(p-2)}{2\beta}} r^{\frac{(n+2)(1-\kappa_2(1-\alpha))}{\beta}}
\end{aligned}
\end{equation*}
for some constant $c=c(n,p, \kappa, C_2, \|\mu_1\|_{L^{\infty}(\OO)})\ge1$. Then we choose $\alpha \in (0,1)$ such that $\frac{(n+2)(1-\kappa_2(1-\alpha))}{\beta} = 0$; that is, $\alpha = \frac{\kappa_2 - 1}{\kappa_2}$.
This and the definition of $\gamma_2$ yield the desired estimate \eqref{compa3-r4}.
\end{proof}

Combining all the previous results, we derive
\begin{proposition}\label{comparison1}
Let $\frac{2n}{n+1}<p\le 2-\frac{1}{n+1}$, $\lambda \ge 1$, $0<\delta<1$, and let $\theta$ be a constant such that $\max\mgh{\frac{n+2}{2(n+1)},\frac{n(2-p)}{2}} < \theta < p-\frac{n}{n+1} \le 1$. Assume that $\gamma$ is given by \eqref{main thm1-r1}. If $u$ and $w$ are weak solutions of \eqref{tweaksol} and \eqref{pem2}, respectively, satisfying \eqref{structure}, $\mu \in \mathcal{L}^{1,\kappa}(\Omega_{T})$ for some $1 < \kappa \leq N$,
\begin{equation*}\label{comparison1-c}
\mint{\iq{4r}}{|Du|^\theta}{dxdt} \le \la^\theta \quad \text{and} \quad \left[\frac{|\mu|(Q_{4r}^{\lambda})}{|Q_{4r}^{\lambda}|}\right]^{\frac{1}{\gamma}} \le \delta\la,
\end{equation*}
then there are constants $\sigma=\sigma(n,\La_0,\La_1,p,\theta)>0$ and $c=c(n,\La_0,\La_1,p,\theta,\kappa,C_0)\ge1$ such that
\begin{equation*}\label{comparison1-r}
\mint{\iq{4r}}{|Du-Dw|^\theta}{dxdt} \leq c \delta^{\sigma_0} \lambda^\theta \quad\text{and}\quad \mint{\iq{2r}}{|Dw|^{p(1+\sigma)}}{dx dt} \le c\lambda^{p(1+\sigma)},
\end{equation*}
where $\sigma_0=\sigma_0(n,p,\theta,\kappa)>0$.
\end{proposition}

\begin{remark}\label{comparison remark}
When we combine Lemmas~\ref{compa1} and \ref{high int}, the range of $p$ in Proposition \ref{comparison1} is valid when $\frac{2n}{n+1}<p\le 2-\frac{1}{n+1}$, not $\max\mgh{\frac{3n+2}{2n+2},\frac{2n}{n+2}}<p\le 2-\frac{1}{n+1}$, since the exponent $\theta$ in Proposition \ref{comparison1} exists only when $p-\frac{n}{n+1}>\max\mgh{\frac{n+2}{2(n+1)},\frac{(2-p)n}{2}}$. Note that $\frac{2n}{n+1} \ge \max\mgh{\frac{3n+2}{2n+2},\frac{2n}{n+2}}$, where the equality holds if and only if $n=2$.
\end{remark}

To prove Theorem~\ref{main thm4}, we need a more comparison estimate as follows. Assume that the vector field $\mathbf{a}$ satisfies \eqref{str1} and a $(\delta,R_0)$-BMO condition for some $R_0>4r$ and $\delta \in (0,1)$.
We consider the unique weak solution $v$ to the coefficient frozen problem
\begin{equation}
\label{pem3}\left\{
\begin{alignedat}{3}
v_t -\ddiv \bar{\ma}_{B_{2r}^{\lambda}}(Dv,t) &= 0 &&\quad \text{in} \ \iq{2r}, \\
v &= w &&\quad \text{on} \ \partial_p \iq{2r},
\end{alignedat}\right.
\end{equation}
where a freezing operator $\bar{\ma}_{B_{2r}^{\lambda}} = \bar{\ma}_{B_{2r}^{\lambda}}(\xi,t) : \bb^n \times (-4r^2,4r^2) \to \bb^n$ is given by
\begin{equation*}
\bar{\mathbf{a}}_{B_{2r}^{\lambda}}(\xi,t) := \mint{B_{2r}^{\lambda}}{\mathbf{a}(\xi,x,t)}{dx}.
\end{equation*}
Then the operator $\bar{\mathbf{a}}_{B_{2r}^{\lambda}}$ satisfies \eqref{str1}.

Now we derive the following comparison result between \eqref{pem2} and \eqref{pem3}:
\begin{lemma}\label{fcompa1}
Let $p>\frac{2n}{n+2}$. Assume that the vector field $\mathbf{a}$ satisfies \eqref{str1} and a $(\delta,R_0)$-BMO condition for some $R_0>4r$ and $\delta \in (0,1)$. If $w$ and $v$ are weak solutions of \eqref{pem2} and \eqref{pem3}, respectively, then there is a constant
$c=c(n,\La_0,\La_1,p)\ge1$ such that
\begin{equation}\label{fcompa1_r}
\mint{\iq{2r}}{|Dw-Dv|^{p}}{dx dt} \le c \delta^{\sigma_1} \la^p  \quad \text{and} \quad \Norm{Dv}_{L^{\infty}(\iq{r})} \le c\lambda,
\end{equation}
where $\sigma_1=\sigma_1(n,\La_0,\La_1,p)>0$.
\end{lemma}

\begin{proof}
The first estimate of \eqref{fcompa1_r} follows from Lemma \ref{high int}, \eqref{2-small BMO2}, and \cite[Lemma 3.10]{BOR13}. For interior regularity results (see \cite{DF85, DF85b, DiB93}),  the second estimate of \eqref{fcompa1_r} holds.
\end{proof}

Finally, we combine Lemmas~\ref{compa1}, \ref{high int}, \ref{compa2} and \ref{fcompa1} to obtain the following regularity estimate:
\begin{proposition}\label{comparison2}
Let $\frac{2n}{n+1}<p\le 2-\frac{1}{n+1}$, $\lambda \ge 1$, and let $\theta$ be a constant such that $\max\mgh{\frac{n+2}{2(n+1)},\frac{n(2-p)}{2}} < \theta < p-\frac{n}{n+1} \le 1$. Assume that the vector field $\mathbf{a}$ satisfies \eqref{str1} and a $(\delta,R_0)$-BMO condition for some $R_0>4r$ and $\delta \in (0,1)$. If $u$, $w$ and $v$ are weak solutions of \eqref{tweaksol}, \eqref{pem2} and \eqref{pem3}, respectively, satisfying $\mu \in \mathcal{L}^{1,\kappa}(\Omega_{T})$ for some $1 < \kappa \leq N$,
\begin{equation*}\label{comparison2-c}
\mint{\iq{4r}}{|Du|^\theta}{dxdt} \le \la^\theta \quad \text{and} \quad \left[\frac{|\mu|(Q_{4r}^{\lambda})}{|Q_{4r}^{\lambda}|}\right]^{\frac{1}{\gamma}} \le \delta\la,
\end{equation*}
where $\gamma$ is given by \eqref{main thm1-r1}, then there is a constant $c_0=c_0(n,\La_0,\La_1,p,\theta,\kappa,C_0)\ge1$ such that
\begin{equation*}\label{comparison2-r}
\mint{\iq{2r}}{|Du-Dv|^\theta}{dxdt} \leq c_0 \delta^{\sigma_2} \lambda^\theta \quad\text{and}\quad \Norm{Dv}_{L^{\infty}(\iq{r})} \le c_0\lambda,
\end{equation*}
where $\sigma_2=\sigma_2(n,\La_0,\La_1,p,\theta,\kappa)>0$.
\end{proposition}

\begin{remark}\label{comparison1-remark}
In the case $\mu = \mu_1 \otimes \mu_2$, where $\mu_1 \in L^{\infty}(\OO)$ and $\mu_2 \in \LL^{1,\kappa_2}(0,T)$ for some $\kappa_2 \in (1,2]$, Propositions~\ref{comparison1} and \ref{comparison2} also hold with $\gamma_2$ replacing $\gamma$, by using Lemma~\ref{compa3} instead of Lemma~\ref{compa2}.
\end{remark}

\section{Proofs of main results}\label{la covering arguments}

In this section, we derive Marcinkiewicz estimates (Theorems~\ref{main thm1}, \ref{main thm3} and \ref{main thm4}) for spatial gradient of a renormalized solution $u$ of the problem \eqref{pem1}. For this, we employ a so-called {\em stopping-time argument} introduced in \cite{AM07}, to obtain decay estimates on the upper-level set of $|Du|$.

We consider a renormalized solution $u$ of \eqref{pem1}. We denote by $u_k:=T_k(u)$ ($k \in \mathbb{N}$) the truncation of $u$ and $\mu_k \in L^{p'}(0,T; W^{-1,p'}(\OO))$ the corresponding measure given in \eqref{tweaksol}.
We also denote by $w_k$ and $v_k$ the corresponding weak solutions of \eqref{pem2} and \eqref{pem3}, respectively.
We know that $\mu_k=\mu_a^+-\mu_a^-+\nu_k^+-\nu_k^-$ for $k\in \mathbb{N}$. Since $\mu_a^\pm+\nu_k^\pm \to \mu_a^\pm+\mu_s^\pm$ tightly as $k \to \infty$, we have 
\begin{equation}\label{est-mu}
\limsup_{k \to \infty} |\mu_k|(K) \leq |\mu|(K) \quad \text{for every compact subset}\  K\subset \OO_T.
\end{equation}

Let $\frac{2n}{n+1} <p \le 2-\frac{1}{n+1}$, let $\theta$ be a constant such that \eqref{ka-range} holds, and take $Q_{2R} \equiv Q_{2R}(y_0,\tau_0) \Subset \OO_T$. Assume that $\mu \in \mathcal{L}^{1,\kappa}(\OO_T)$ for some $\kappa \in (1,N]$, where $N:=n+2$.
We consider a parameter $\lambda_0$ to be defined, such that
\begin{equation}\label{la0}
  \lambda_0^{\frac{1}{d}} := \gh{\mint{Q_{2R}}{|Du|^{\theta}}{dxdt}}^{\frac{1}{\theta}} + \frac{1}{\delta} \bgh{\frac{|\mu|(Q_{2R})}{|Q_{2R}|}}^{\frac{1}{\gamma}} + 1,
\end{equation}
where the constant $d$ is given by \eqref{scaling deficit} and $\gamma$ is given by \eqref{main thm1-r1}.  The number $\delta \in (0,1)$ will be determined later as a universal constant depending only on $n$, $\La_0$, $\La_1$, $p$, $\theta$, $\kappa$ and $C_0$.

\subsection{Stopping-time arguments}\label{Stopping time argument}

For $\varLambda > \la_0$ and $r \in (0,2R]$, we define
\begin{equation*}\label{up:E}
  E(r,\varLambda):=\mgh{z\in Q_{r} : |Du(z)|>\varLambda}.
\end{equation*}
For fixed radii $R \le R_1 < R_2 \le 2R$, the relation $Q_{r}^{\lambda}(z_0) \subset Q_{R_2} \subset Q_{2R}$
holds whenever $z_0 \in Q_{R_1}$, $r \in (0,R_2-R_1]$ and $\lambda \in [\la_0,\infty)$. Fix
$
z_0 \in E(R_1,4\lambda).
$
For almost every such point, Lebesgue's differentiation theorem implies
\begin{equation}\label{stop1}
  \lim_{s\searrow0} \bgh{\gh{\mint{Q_s^{\lambda}(z_0)}{|Du|^{\theta}}{dxdt}}^{\frac{1}{\theta}} +\frac{1}{\delta}\bgh{\frac{|\mu|(\pc{Q_s^{\lambda}(z_0)})}{|Q_s^{\lambda}(z_0)|}}^{\frac{1}{\gamma}}} \ge \abs{Du(z_0)} > 4\lambda,
\end{equation}
where the symbol $\pc{Q}$ denotes the parabolic closure of $Q$ defined as $\pc{Q} := Q \cup \partial_p Q$.
We consider
\begin{equation}\label{stop2}
  \la > B\la_0, \ \ \text{where} \ B:=\gh{\frac{320R}{R_2-R_1}}^{\frac{dN}{\theta}} > 1.
\end{equation}
For any radius $s$ with
\begin{equation}\label{stop3}
  \frac{R_2-R_1}{160} \le s \le \frac{R_2-R_1}{2},
\end{equation}
we see from \eqref{la0}, \eqref{stop2}, \eqref{stop3} and \eqref{scaling deficit} that
\begin{equation}\label{stop4}
  \begin{aligned}
  &\gh{\mint{Q_s^{\lambda}(z_0)}{|Du|^{\theta}}{dxdt}}^{\frac{1}{\theta}} +\frac{1}{\delta}\bgh{\frac{|\mu|(\pc{Q_s^{\lambda}(z_0)})}{|Q_s^{\lambda}(z_0)|}}^{\frac{1}{\gamma}}\\ &\qquad \le \gh{\frac{|Q_{2R}|}{|Q_s^{\lambda}|}}^{\frac{1}{\theta}} \gh{\mint{Q_{2R}}{|Du|^{\theta}}{dxdt}}^{\frac{1}{\theta}} + \frac{1}{\delta} \gh{\frac{|Q_{2R}|}{|Q_s^{\lambda}|}}^{\frac{1}{\gamma}} \bgh{\frac{|\mu|(Q_{2R})}{|Q_{2R}|}}^{\frac{1}{\gamma}}\\
  &\qquad \le \gh{\frac{2R}{s}}^{\frac{N}{\theta}} \lambda^{\frac{n(2-p)}{2\theta}}\la_0^{\frac{1}{d}}
  \le \gh{\frac{320R}{R_2-R_1}}^{\frac{N}{\theta}}\la_0^{\frac{1}{d}} \la^{\frac{n(2-p)}{2\theta}}\\
  &\qquad = (B\la_0)^{\frac{1}{d}} \la^{\frac{n(2-p)}{2\theta}} < \lambda < 4\lambda,
  \end{aligned}
\end{equation}
where we used the inequality $\theta<\gamma$, see \eqref{ka-range} and  Remark~\ref{main rmk1}\,\eqref{R:m2}.
According to \eqref{stop1}, \eqref{stop4} and the (absolute) continuity of the integral and the measure, there exists a maximal radius $r_{z_0} \in \gh{0,\frac{R_2-R_1}{160}}$ such that
\begin{equation}\label{stop5}
  \gh{\mint{Q_{r_{z_0}}^{\lambda}(z_0)}{|Du|^{\theta}}{dxdt}}^{\frac{1}{\theta}} + \frac{1}{\delta} \bgh{\frac{|\mu|(\pc{Q_{r_{z_0}}^{\lambda}(z_0)})}{|Q_{r_{z_0}}^{\lambda}(z_0)|}}^{\frac{1}{\gamma}} = 4\lambda
\end{equation}
and
\begin{equation}\label{stop6}
  \gh{\mint{Q_s^{\lambda}(z_0)}{|Du|^{\theta}}{dxdt}}^{\frac{1}{\theta}} + \frac{1}{\delta} \bgh{\frac{|\mu|(\pc{Q_s^{\lambda}(z_0)})}{|Q_s^{\lambda}(z_0)|}}^{\frac{1}{\gamma}} < 4\lambda \ \  \text{for any} \ s \in \left(r_{z_0},\frac{R_2-R_1}{2}\right].
\end{equation}

\subsection{Decay estimates}\label{Estimates on upper-level sets}

The goal of this subsection is to derive a decay estimate on an upper-level set of $|Du|$, see \eqref{est12} below.
Let $z_0 \in E(R_1,4\lambda)$, let $r_{z_0} \in \gh{0,\frac{R_2-R_1}{160}}$ be a maximal radius as in \eqref{stop5}. For $\la>B\la_0$, the upper-level set $E(R_1,4\lambda)$ can be covered by a family $\mathcal{F} \equiv \mgh{Q_{4r_{z_0}}^{\lambda}(z_0)}_{z_0 \in E(R_1,4\lambda)}$.  By the standard Vitali covering lemma (see e.g. \cite[Theorem C.1]{Bog07} or \cite[Theorem 1.24]{EG15}), there exists a countable subfamily $\mgh{Q_{4r_{z_i}}^{\lambda}(z_i)}_{i\in \mathbb{N}} \subset \mathcal{F}$ consisting of pairwise disjoint cylinders such that
\begin{equation*}\label{est1}
  E(R_1,4\lambda) \setminus \mathcal{N} \subset \bigcup_{i=1}^\infty Q_{20r_{z_i}}^{\lambda}(z_i) \subset Q_{R_2},
\end{equation*}
where $\mathcal{N}$ is a Lebesgue measure zero set; that is, $|\mathcal{N}|=0$. For simplicity, we denote
\begin{gather*}
  Q_i^0 := Q_{r_{z_i}}^{\lambda}(z_i), \ Q_i^1 := Q_{4r_{z_i}}^{\lambda}(z_i), \ Q_i^2 := Q_{20r_{z_i}}^{\lambda}(z_i),\\
  Q_i^3 := Q_{40r_{z_i}}^{\lambda}(z_i), \  \text{and} \  Q_i^4 := Q_{80r_{z_i}}^{\lambda}(z_i).
\end{gather*}
Note that since $160r_{z_i} < R_2-R_1 \le R$, we have $Q_i^4 \subset Q_{R_2} \subset Q_{2R}$.

We now fix $H\ge4$ to be chosen later and we estimate
\begin{equation}\label{est2}
  |E(R_1,H\lambda)| \le \sum_{i=1}^\infty \abs{Q_i^2 \cap E(R_2,H\lambda)}.
\end{equation}
We first split into
\begin{equation}\label{est3}
\begin{aligned}
  \abs{Q_i^2 \cap E(R_2,H\lambda)} &= \abs{\mgh{z\in Q_i^2 : |Du|>H\lambda}}\\
  &\le \abs{\mgh{z\in Q_i^2 : |Du-Du_{k}|>\frac{H\lambda}{3}}}\\
  &\qquad + \abs{\mgh{z\in Q_i^2 : |Du_{k}-Dw_{k,i}|>\frac{H\lambda}{3}}}\\
  &\qquad + \abs{\mgh{z\in Q_i^2 : |Dw_{k,i}|>\frac{H\lambda}{3}}}\\
  &=: I_1 + I_2 + I_3,
\end{aligned}
\end{equation}
where $w_{k,i}$ is the weak solution of the Cauchy-Dirichlet problem
\begin{equation}
\label{E:w_i}\left\{
\begin{alignedat}{3}
\partial_{t}w_{k,i} - \ddiv\mathbf{a}(Dw_{k,i},x,t) &= 0 &&\quad \text{in} \ Q_i^4, \\
w_{k,i} &= u_k &&\quad \text{on} \ \partial_p Q_i^4.
\end{alignedat}\right.
\end{equation}
From the absolute continuity of the Lebesgue integral and \eqref{stop6}, for each $\varepsilon \in (0,1)$ we have
\begin{equation}\label{est3.5}
\begin{aligned}
  I_1 &\le \frac{3^{\theta}|Q_i^2|}{(H\lambda)^{\theta}}\mint{Q_i^2}{|Du-Du_{k}|^{\theta}}{dxdt} = \frac{3^{\theta}|Q_i^2|}{(H\lambda)^{\theta}}\mint{Q_i^2}{\chi_{\{|Du|>k\}}|Du|^{\theta}}{dxdt}\\ &\le \frac{c\varepsilon}{H^{\theta}} |Q_i^2|
\end{aligned}
\end{equation}
for $k$ large enough. Moreover, applying Proposition~\ref{comparison1} with $r=20r_{z_i}$ and using \eqref{est-mu} and \eqref{stop6}, we deduce
\begin{equation}\label{est4}
  I_2 \le \frac{3^{\theta}}{(H\lambda)^{\theta}}\integral{Q_i^2}{|Du_k-Dw_{k,i}|^{\theta}}{dxdt}
  \le \frac{c\delta^{\theta\sigma_0}}{H^{\theta}} |Q_i^2|,
\end{equation}
where $\sigma_0=\sigma_0(n,p,\theta,\kappa)>0$, and
\begin{equation}\label{est5.95}
\begin{aligned}
  I_3 \le \gh{\frac{3}{H\la}}^{p(1+\sigma)} \integral{Q_i^2}{|Dw_{k,i}|^{p(1+\sigma)}}{dxdt} \le \frac{c}{H^{p(1+\sigma)}} |Q_i^2|
\end{aligned}
\end{equation}
for some two constants $c=c(n,\La_0,\La_1,p,\theta,\kappa,C_0)\ge1$.

Plugging \eqref{est3.5}--\eqref{est5.95} into \eqref{est3}, we obtain
\begin{equation}\label{est5.97}
  \abs{Q_i^2 \cap E(R_2,H\lambda)} \le \gh{\frac{c\varepsilon}{H^{\theta}} + \frac{c\delta^{\theta\sigma_0}}{H^{\theta}} + \frac{c}{H^{p(1+\sigma)}}} |Q_i^2|.
\end{equation}
Now we will estimate $|Q_i^2|$. Recalling \eqref{stop5}, we have then either
\begin{equation}\label{est6}
  2\lambda \le \gh{\mint{Q_i^0}{|Du|^{\theta}}{dxdt}}^{\frac{1}{\theta}} \quad  \text{or} \quad 2\lambda \le \frac{1}{\delta} \bgh{\frac{|\mu|(\pc{Q_i^0})}{|Q_i^0|}}^{\frac{1}{\gamma}}.
\end{equation}
We assume that the first case of \eqref{est6} holds. Then it follows
\begin{equation}\label{est7}
\begin{aligned}
  (2^\theta -1) \la^{\theta} &\le \frac{1}{|Q_i^0|}\integral{Q_i^0 \cap \mgh{|Du|>\lambda}}{|Du|^{\theta}}{dxdt}\\
  &\le \gh{\mint{Q_i^0}{|Du|^{\tilde{\theta}}}{dxdt}}^{\frac{\theta}{\tilde{\theta}}} \gh{\frac{\abs{Q_i^0 \cap \mgh{|Du|>\lambda}}}{\abs{Q_i^0}}}^{1-\frac{\theta}{\tilde{\theta}}}
\end{aligned}
\end{equation}
for any $\tilde{\theta} \in \gh{\theta,p-\frac{n}{n+1}}$.
Applying Proposition~\ref{comparison1} with $\tilde{\theta}$ instead of $\theta$ and utilizing \eqref{est-mu} and \eqref{stop6}, we deduce
\begin{equation*}
\begin{aligned}
  \mint{Q_i^0}{|Du|^{\tilde{\theta}}}{dxdt} &\le \mint{Q_i^0}{|Du-Du_k|^{\tilde{\theta}}}{dxdt} +  \mint{Q_i^0}{|Du_k-Dw_{k,i}|^{\tilde{\theta}}}{dxdt}\\
  &\qquad + \mint{Q_i^0}{|Dw_{k,i}|^{\tilde{\theta}}}{dxdt}\\
  &\le c \la^{\tilde{\theta}}
\end{aligned}
\end{equation*}
for $k$ large enough. Inserting this estimate into \eqref{est7}, we obtain
\begin{equation}\label{est8}
  \abs{Q_i^0} \le c \abs{Q_i^0 \cap \mgh{|Du|>\lambda}}.
\end{equation}
If the second case of \eqref{est6} holds, then we see
\begin{equation}\label{est9}
|Q_i^0| \le \frac{|\mu|(\pc{Q_i^0})}{\gh{2\delta \lambda}^\gamma}.
\end{equation}
Assertions \eqref{est8} and \eqref{est9} yield
\begin{equation}\label{est10}
  \abs{Q_i^2} = 20^N \abs{Q_i^0} \le c \abs{Q_i^1 \cap E\gh{R_2,\lambda}} + \frac{c|\mu|(\pc{Q_i^0})}{\gh{\delta \lambda}^\gamma}.
\end{equation}
We combine \eqref{est5.97} and \eqref{est10} to obtain
\begin{equation}\label{est11}
  \abs{Q_i^2 \cap E(R_2,H\lambda)} \le \gh{\frac{c\varepsilon}{H^{\theta}} + \frac{c\delta^{\theta\sigma_0}}{H^{\theta}} + \frac{c}{H^{p(1+\sigma)}}} \abs{Q_i^1 \cap E\gh{R_2,\lambda}} + \frac{c|\mu|(Q_i^1)}{\gh{\delta \lambda}^\gamma}.
\end{equation}

Since the cylinders $\mgh{Q_i^1}$ are pairwise disjoint, we have from \eqref{est2} and \eqref{est11} that
\begin{equation}\label{est12}
  \abs{E(R_1,H\lambda)} \le \gh{\frac{c\varepsilon}{H^{\theta}} + \frac{c\delta^{\theta\sigma_0}}{H^{\theta}} + \frac{c}{H^{p(1+\sigma)}}} \abs{E\gh{R_2,\lambda}} + \frac{c|\mu|(Q_{2R})}{\gh{\delta \lambda}^\gamma}
\end{equation}
for some constant $c=c(n,\La_0,\La_1,p,\theta,\kappa,C_0)\ge1$.

\subsection{Marcinkiewicz estimates}\label{Marcinkiewicz estimates}
Before Theorems~\ref{main thm1}, \ref{main thm3} and \ref{main thm4}, we introduce the following technical assertion:
\begin{lemma}[See {\cite[Lemma 6.1]{Giu03}}]\label{L:Aux01}
Let $\phi:[r,\rho] \to \mathbb{R}_{\geq 0}$ be a nonnegative bounded function. Assume that for $r\leq t < s \leq \rho$ we have
\[
\phi(t) \leq \vartheta \phi(s) + A (s - t)^{-\beta} + C
\]	
with $0\le \vartheta <1$, $A, C \geq 0$, and $\beta >0$. Then there holds
\[
\phi(r) \leq c(\beta,\vartheta) \left[A (\rho - r)^{-\beta} + C\right].
\]
\end{lemma}

Now we prove Theorems~\ref{main thm1}, \ref{main thm3} and \ref{main thm4}.
\begin{proof}[Proof of Theorem~\ref{main thm1}]
For $r \in (0,2R]$, we define the upper-level set
\begin{equation*}
  E_l(r,\la):=\mgh{z\in Q_{r} : T_l(|Du|) > \la},
\end{equation*}
where $T_l$ is the truncation operator \eqref{T_k}. Note that $E_l(r,\la) = E(r,\la)$ for $l>\la$.
Then it follows from \eqref{est12} that
\begin{equation}\label{est12.5}
  \abs{E_l(R_1,H\lambda)} \le \gh{\frac{c\varepsilon}{H^{\theta}} + \frac{c\delta^{\theta\sigma_0}}{H^{\theta}} + \frac{c}{H^{p(1+\sigma)}}} \abs{E_l\gh{R_2,\lambda}} + \frac{c|\mu|(Q_{2R})}{\gh{\delta \lambda}^\gamma}
\end{equation}
whenever $l>H\lambda$, where $\sigma_0=\sigma_0(n,p,\theta,\kappa)>0$.
Multiplying \eqref{est12.5} by $(H\lambda)^{\gamma}$, we have
\begin{equation*}\label{est13}
\begin{aligned}
  (H\lambda)^{\gamma}\abs{E_l(R_1,H\lambda)} &\le \gh{\frac{c\varepsilon}{H^{\theta-\gamma}} + \frac{c\delta^{\theta\sigma_0}}{H^{\theta-\gamma}} + \frac{c}{H^{p(1+\sigma)-\gamma}}} \la^{\gamma} \abs{E_l\gh{R_2,\lambda}}\\
  &\qquad + \frac{cH^\gamma}{\delta^\gamma} |\mu|(Q_{2R})
\end{aligned}
\end{equation*}
for some constant $c=c(n,\La_0,\La_1,p,\theta,\kappa,C_0)\ge1$.
First choose $H$ sufficiently large such that
\begin{equation}\label{est13.5}
  \frac{c}{H^{p(1+\sigma)-\gamma}} \le \frac{1}{4} \quad \text{and} \quad H\ge4,
\end{equation}
and then choose $\varepsilon$ and $\delta$ sufficiently small such that
\begin{equation*}
  \frac{c\varepsilon}{H^{\theta-\gamma}} \le \frac{1}{4} \quad \text{and} \quad \frac{c\delta^{\theta\sigma_0}}{H^{\theta-\gamma}} \le \frac{1}{4}.
\end{equation*}
From \eqref{est13.5}, $\gamma$ have to be chosen such that $\gamma < p(1+\sigma)$. Thus, we choose the critical value $\kappa_c$ such that $\gamma = p(1+\sigma)$; that is,
$$
\frac{\kappa_c}{\kappa_c-1}\max\mgh{p-1, \ \frac{1}{2}\gh{p-\frac{n(2-p)}{\kappa_c}}} = p(1+\sigma).
$$
Taking the supremum with respect to $\la > B\la_0$, we obtain
\begin{equation*}\label{est14}
  \sup_{\la > HB\la_0} \la^{\gamma}\abs{E_l(R_1,\lambda)} \le \frac{3}{4} \sup_{\la > B\la_0} \la^{\gamma}\abs{E_l(R_2,\lambda)} + c |\mu|(Q_{2R}),
\end{equation*}
whenever $l>H\lambda$. Here $\la_0$ is given by \eqref{la0}.
Recalling \eqref{Marcinkiewicz space}, we see from \eqref{stop2} that
\begin{equation*}\label{est15}
\begin{aligned}
  \Abs{T_l(|Du|)}_{\M^\gamma(Q_{R_1},\bb^n)}^\gamma &\le \frac{3}{4} \Abs{T_l(|Du|)}_{\M^\gamma(Q_{R_2},\bb^n)}^\gamma + c |\mu|(Q_{2R}) + (HB\la_0)^{\gamma}|Q_{R_1}|\\
  &\le \frac{3}{4} \Abs{T_l(|Du|)}_{\M^\gamma(Q_{R_2},\bb^n)}^\gamma + c |\mu|(Q_{2R})+ c\la_0^{\gamma}\gh{\frac{R}{R_2-R_1}}^{\frac{\gamma dN}{\theta}} R^N
\end{aligned}
\end{equation*}
for all $R\le R_1 < R_2 \le 2R$.
Applying Lemma~\ref{L:Aux01}  and letting $l\to\infty$, we discover
\begin{equation*}\label{est16}
\begin{aligned}
  \Abs{Du}_{\M^{\gamma}(Q_{R},\bb^n)}^\gamma &\le c |\mu|(Q_{2R}) + c \la_0^{\gamma}R^N\\
  &\le c R^N \frac{|\mu|(Q_{2R})}{|Q_{2R}|} + c R^N + c R^N \gh{\mint{Q_{2R}}{|Du|^{\theta}}{dxdt}}^{\frac{d\gamma}{\theta}}\\
  &\qquad + c R^N \bgh{\frac{|\mu|(Q_{2R})}{|Q_{2R}|}}^{d},
\end{aligned}
\end{equation*}
which completes the proof.
\end{proof}

\begin{proof}[Proof of Theorem~\ref{main thm3}]
In view of Lemma~\ref{compa3} and Remark~\ref{comparison1-remark}, we replace $\gamma$ by $\gamma_2$ in Sections~\ref{Stopping time argument}--\ref{Estimates on upper-level sets}.
Proceeding as in Sections~\ref{Stopping time argument}--\ref{Estimates on upper-level sets} and Proof of Theorem~\ref{main thm1} above, we can obtain Theorem~\ref{main thm3}.
\end{proof}

\begin{proof}[Proof of Theorem~\ref{main thm4}]
We proceed as in Section~\ref{la covering arguments}. We only mention the parts that change in Section~\ref{la covering arguments}.
We choose $H\ge \max\{4,3c_0\}$, where the constant $c_0$ is given by Proposition~\ref{comparison2}. Instead of \eqref{est3}, we split into
\begin{equation*}\label{w:est3}
\begin{aligned}
  \abs{Q_i^2 \cap E(R_2,H\la)} &= \abs{\mgh{z\in Q_i^2 : |Du|>H\la}}\\
  &\le \abs{\mgh{z\in Q_i^2 : |Du-Du_{k}|>\frac{H\la}{3}}}\\
  &\qquad + \abs{\mgh{z\in Q_i^2 : |Du_{k}-Dv_{k,i}|>\frac{H\la}{3}}}\\
  &\qquad + \abs{\mgh{z\in Q_i^2 : |Dv_{k,i}|>\frac{H\la}{3}}}\\
  &=: J_1 + J_2 + J_3,
\end{aligned}
\end{equation*}
where $v_{k,i}$ is the weak solution of the Cauchy-Dirichlet problem
\begin{equation*}
\label{E:v_i}\left\{
\begin{alignedat}{3}
\partial_{t}v_{k,i} - \ddiv\bar{\ma}_{B_{40r_{z_i}}^{\lambda}}(Dv_{k,i},t) &= 0 &&\quad \text{in} \ Q_i^3, \\
v_{k,i} &= w_{k,i} &&\quad \text{on} \ \partial_p Q_i^3.
\end{alignedat}\right.
\end{equation*}
Here $w_{k,i}$ is the weak solution of \eqref{E:w_i}. Then we see from Proposition~\ref{comparison2} and the choice of $H$ that $J_3=0$. Also, we can estimate $J_1$ and $J_2$ similar to the estimates of $I_1$ and $I_2$ in Section~\ref{Estimates on upper-level sets}. Performing the rest of Section~\ref{Estimates on upper-level sets}, we obtain, instead of \eqref{est12}, the following decay estimate
\begin{equation}\label{w:est12}
  \abs{E(R_1,H\lambda)} \le \gh{\frac{c\varepsilon}{H^{\theta}} + \frac{c\delta^{\sigma_2}}{H^{\theta}}} \abs{E\gh{R_2,\lambda}} + \frac{c|\mu|(Q_{2R})}{\gh{\delta \lambda}^\gamma},
\end{equation}
where $\sigma_2=\sigma_2(n,\La_0,\La_1,p,\theta,\kappa)>0$.
Proceeding as in Proof of Theorem~\ref{main thm1} above with \eqref{w:est12} replacing \eqref{est12.5}, we deduce Theorem~\ref{main thm4}. We remark that we obtain Theorem~\ref{main thm4} without \eqref{est13.5}; that is, $\gamma<\infty$.
\end{proof}


\section*{Acknowledgments}
The author thanks Pilsoo Shin for helpful comments on an earlier draft of this paper. The author also thanks the anonymous referees for the valuable comments, which improved the exposition and the accuracy of the paper.
This paper was supported by the National Research Foundation of Korea grant (No. NRF-2019R1C1C1003844) from the Korea government and by Education and Research promotion program of KOREATECH in 2022.



\begin{bibdiv}
\begin{biblist}

\bib{AM07}{article}{
      author={Acerbi, E.},
      author={Mingione, G.},
       title={Gradient estimates for a class of parabolic systems},
        date={2007},
        ISSN={0012-7094},
     journal={Duke Math. J.},
      volume={136},
      number={2},
       pages={285\ndash 320},
         url={http://dx.doi.org/10.1215/S0012-7094-07-13623-8},
      review={\MR{2286632}},
}

\bib{AMST99}{article}{
      author={Andreu, F.},
      author={Maz\'{o}n, J.~M.},
      author={Segura~de Le\'{o}n, S.},
      author={Toledo, J.},
       title={Existence and uniqueness for a degenerate parabolic equation with
  {$L^1$}-data},
        date={1999},
        ISSN={0002-9947},
     journal={Trans. Amer. Math. Soc.},
      volume={351},
      number={1},
       pages={285\ndash 306},
         url={https://doi.org/10.1090/S0002-9947-99-01981-9},
      review={\MR{1433108}},
}

\bib{AKP15}{article}{
      author={Avelin, B.},
      author={Kuusi, T.},
      author={Parviainen, M.},
       title={Variational parabolic capacity},
        date={2015},
        ISSN={1078-0947},
     journal={Discrete Contin. Dyn. Syst.},
      volume={35},
      number={12},
       pages={5665\ndash 5688},
         url={https://doi.org/10.3934/dcds.2015.35.5665},
      review={\MR{3393250}},
}

\bib{Bar14}{article}{
      author={Baroni, P.},
       title={Marcinkiewicz estimates for degenerate parabolic equations with
  measure data},
        date={2014},
        ISSN={0022-1236},
     journal={J. Funct. Anal.},
      volume={267},
      number={9},
       pages={3397\ndash 3426},
         url={https://doi.org/10.1016/j.jfa.2014.08.017},
      review={\MR{3261114}},
}

\bib{Bar14b}{article}{
      author={Baroni, P.},
       title={Nonlinear parabolic equations with {M}orrey data},
        date={2014},
        ISSN={0035-6298},
     journal={Riv. Math. Univ. Parma (N.S.)},
      volume={5},
      number={1},
       pages={65\ndash 92},
      review={\MR{3289597}},
}

\bib{Bar17}{article}{
      author={Baroni, P.},
       title={Singular parabolic equations, measures satisfying density
  conditions, and gradient integrability},
        date={2017},
        ISSN={0362-546X},
     journal={Nonlinear Anal.},
      volume={153},
       pages={89\ndash 116},
         url={https://doi.org/10.1016/j.na.2016.10.019},
      review={\MR{3614663}},
}

\bib{BH12}{article}{
      author={Baroni, P.},
      author={Habermann, J.},
       title={Calder\'{o}n-{Z}ygmund estimates for parabolic measure data
  equations},
        date={2012},
        ISSN={0022-0396},
     journal={J. Differential Equations},
      volume={252},
      number={1},
       pages={412\ndash 447},
         url={https://doi.org/10.1016/j.jde.2011.08.016},
      review={\MR{2852212}},
}

\bib{BDGO97}{article}{
      author={Boccardo, L.},
      author={Dall'Aglio, A.},
      author={Gallou\"et, T.},
      author={Orsina, L.},
       title={Nonlinear parabolic equations with measure data},
        date={1997},
        ISSN={0022-1236},
     journal={J. Funct. Anal.},
      volume={147},
      number={1},
       pages={237\ndash 258},
         url={http://dx.doi.org/10.1006/jfan.1996.3040},
      review={\MR{1453181}},
}

\bib{BG89}{article}{
      author={Boccardo, L.},
      author={Gallou\"et, T.},
       title={Nonlinear elliptic and parabolic equations involving measure
  data},
        date={1989},
        ISSN={0022-1236},
     journal={J. Funct. Anal.},
      volume={87},
      number={1},
       pages={149\ndash 169},
         url={http://dx.doi.org/10.1016/0022-1236(89)90005-0},
      review={\MR{1025884}},
}

\bib{Bog07}{book}{
      author={B\"ogelein, V.},
       title={Regularity results for weak and very weak solutions of higher
  order parabolic systems},
   publisher={Ph.D. Thesis},
        date={2007},
}

\bib{BD18}{article}{
      author={Bui, T.~A.},
      author={Duong, X.~T.},
       title={Global {M}arcinkiewicz estimates for nonlinear parabolic
  equations with nonsmooth coefficients},
        date={2018},
        ISSN={0391-173X},
     journal={Ann. Sc. Norm. Super. Pisa Cl. Sci. (5)},
      volume={18},
      number={3},
       pages={881\ndash 916},
      review={\MR{3807590}},
}

\bib{BCS21}{article}{
      author={Byun, S.-S.},
      author={Cho, N.},
      author={Song, K.},
       title={Optimal fractional differentiability for nonlinear parabolic
  measure data problems},
        date={2021},
        ISSN={0893-9659},
     journal={Appl. Math. Lett.},
      volume={112},
       pages={106816, 10 pp},
         url={https://doi.org/10.1016/j.aml.2020.106816},
      review={\MR{4162964}},
}

\bib{BOR13}{article}{
      author={Byun, S.-S.},
      author={Ok, J.},
      author={Ryu, S.},
       title={Global gradient estimates for general nonlinear parabolic
  equations in nonsmooth domains},
        date={2013},
        ISSN={0022-0396},
     journal={J. Differential Equations},
      volume={254},
      number={11},
       pages={4290\ndash 4326},
         url={http://dx.doi.org/10.1016/j.jde.2013.03.004},
      review={\MR{3035434}},
}

\bib{BP18b}{article}{
      author={Byun, S.-S.},
      author={Park, J.-T.},
       title={Global weighted {O}rlicz estimates for parabolic measure data
  problems: application to estimates in variable exponent spaces},
        date={2018},
        ISSN={0022-247X},
     journal={J. Math. Anal. Appl.},
      volume={467},
      number={2},
       pages={1194\ndash 1207},
         url={https://doi.org/10.1016/j.jmaa.2018.07.059},
      review={\MR{3842429}},
}

\bib{BPS21}{article}{
      author={Byun, S.-S.},
      author={Park, J.-T.},
      author={Shin, P.},
       title={Global regularity for degenerate/singular parabolic equations
  involving measure data},
        date={2021},
        ISSN={0944-2669},
     journal={Calc. Var. Partial Differential Equations},
      volume={60},
      number={1},
       pages={Paper No. 18, 32 pp},
         url={https://doi.org/10.1007/s00526-020-01906-2},
      review={\MR{4201641}},
}

\bib{DiB93}{book}{
      author={DiBenedetto, E.},
       title={Degenerate parabolic equations},
      series={Universitext},
   publisher={Springer-Verlag, New York},
        date={1993},
        ISBN={0-387-94020-0},
         url={https://doi.org/10.1007/978-1-4612-0895-2},
      review={\MR{1230384}},
}

\bib{DF85}{article}{
      author={DiBenedetto, E.},
      author={Friedman, A.},
       title={H\"{o}lder estimates for nonlinear degenerate parabolic systems},
        date={1985},
        ISSN={0075-4102},
     journal={J. Reine Angew. Math.},
      volume={357},
       pages={1\ndash 22},
         url={https://doi.org/10.1515/crll.1985.357.1},
      review={\MR{783531}},
}

\bib{DF85b}{article}{
      author={DiBenedetto, E.},
      author={Friedman, A.},
       title={Addendum to: ``{H}\"{o}lder estimates for nonlinear degenerate
  parabolic systems''},
        date={1985},
        ISSN={0075-4102},
     journal={J. Reine Angew. Math.},
      volume={363},
       pages={217\ndash 220},
         url={https://doi.org/10.1515/crll.1985.363.217},
      review={\MR{814022}},
}

\bib{DZ21}{misc}{
      author={Dong, H.},
      author={Zhu, H.},
       title={Gradient estimates for singular $p$-laplace type equations with
  measure data},
        date={2021},
        note={\href{https://arxiv.org/abs/2102.08584}{arXiv:2102.08584}},
}

\bib{DPP03}{article}{
      author={Droniou, J.},
      author={Porretta, A.},
      author={Prignet, A.},
       title={Parabolic capacity and soft measures for nonlinear equations},
        date={2003},
        ISSN={0926-2601},
     journal={Potential Anal.},
      volume={19},
      number={2},
       pages={99\ndash 161},
         url={https://doi.org/10.1023/A:1023248531928},
      review={\MR{1976292}},
}

\bib{DM11}{article}{
      author={Duzaar, F.},
      author={Mingione, G.},
       title={Gradient estimates via non-linear potentials},
        date={2011},
        ISSN={0002-9327},
     journal={Amer. J. Math.},
      volume={133},
      number={4},
       pages={1093\ndash 1149},
         url={http://dx.doi.org/10.1353/ajm.2011.0023},
      review={\MR{2823872}},
}

\bib{EG15}{book}{
      author={Evans, L.~C.},
      author={Gariepy, R.~F.},
       title={Measure theory and fine properties of functions},
     edition={Revised},
      series={Textbooks in Mathematics},
   publisher={CRC Press, Boca Raton, FL},
        date={2015},
        ISBN={978-1-4822-4238-6},
      review={\MR{3409135}},
}

\bib{FST91}{article}{
      author={Fukushima, M.},
      author={Sato, K.-i.},
      author={Taniguchi, S.},
       title={On the closable parts of pre-{D}irichlet forms and the fine
  supports of underlying measures},
        date={1991},
        ISSN={0030-6126},
     journal={Osaka J. Math.},
      volume={28},
      number={3},
       pages={517\ndash 535},
         url={http://projecteuclid.org/euclid.ojm/1200783223},
      review={\MR{1144471}},
}

\bib{Giu03}{book}{
      author={Giusti, E.},
       title={Direct methods in the calculus of variations},
   publisher={World Scientific Publishing Co., Inc., River Edge, NJ},
        date={2003},
        ISBN={981-238-043-4},
         url={http://dx.doi.org/10.1142/9789812795557},
      review={\MR{1962933}},
}

\bib{Gra14}{book}{
      author={Grafakos, L.},
       title={Classical {F}ourier analysis},
     edition={Third},
      series={Graduate Texts in Mathematics},
   publisher={Springer, New York},
        date={2014},
      volume={249},
        ISBN={978-1-4939-1193-6; 978-1-4939-1194-3},
         url={http://dx.doi.org/10.1007/978-1-4939-1194-3},
      review={\MR{3243734}},
}

\bib{KKKP13}{article}{
      author={Kinnunen, J.},
      author={Korte, R.},
      author={Kuusi, T.},
      author={Parviainen, M.},
       title={Nonlinear parabolic capacity and polar sets of superparabolic
  functions},
        date={2013},
        ISSN={0025-5831},
     journal={Math. Ann.},
      volume={355},
      number={4},
       pages={1349\ndash 1381},
         url={https://doi.org/10.1007/s00208-012-0825-x},
      review={\MR{3037018}},
}

\bib{KL00}{article}{
      author={Kinnunen, J.},
      author={Lewis, J.~L.},
       title={Higher integrability for parabolic systems of {$p$}-{L}aplacian
  type},
        date={2000},
        ISSN={0012-7094},
     journal={Duke Math. J.},
      volume={102},
      number={2},
       pages={253\ndash 271},
         url={http://dx.doi.org/10.1215/S0012-7094-00-10223-2},
      review={\MR{1749438}},
}

\bib{KR19}{article}{
      author={Klimsiak, T.},
      author={Rozkosz, A.},
       title={On the structure of diffuse measures for parabolic capacities},
        date={2019},
        ISSN={1631-073X},
     journal={C. R. Math. Acad. Sci. Paris},
      volume={357},
      number={5},
       pages={443\ndash 449},
         url={https://doi.org/10.1016/j.crma.2019.04.012},
      review={\MR{3959840}},
}

\bib{KM13b}{article}{
      author={Kuusi, T.},
      author={Mingione, G.},
       title={Gradient regularity for nonlinear parabolic equations},
        date={2013},
        ISSN={0391-173X},
     journal={Ann. Sc. Norm. Super. Pisa Cl. Sci. (5)},
      volume={12},
      number={4},
       pages={755\ndash 822},
      review={\MR{3184569}},
}

\bib{KM14c}{article}{
      author={Kuusi, T.},
      author={Mingione, G.},
       title={Riesz potentials and nonlinear parabolic equations},
        date={2014},
        ISSN={0003-9527},
     journal={Arch. Ration. Mech. Anal.},
      volume={212},
      number={3},
       pages={727\ndash 780},
         url={https://doi.org/10.1007/s00205-013-0695-8},
      review={\MR{3187676}},
}

\bib{KM14b}{article}{
      author={Kuusi, T.},
      author={Mingione, G.},
       title={The {W}olff gradient bound for degenerate parabolic equations},
        date={2014},
        ISSN={1435-9855},
     journal={J. Eur. Math. Soc. (JEMS)},
      volume={16},
      number={4},
       pages={835\ndash 892},
         url={http://dx.doi.org/10.4171/JEMS/449},
      review={\MR{3191979}},
}

\bib{Min07}{article}{
      author={Mingione, G.},
       title={The {C}alder\'on-{Z}ygmund theory for elliptic problems with
  measure data},
        date={2007},
        ISSN={0391-173X},
     journal={Ann. Sc. Norm. Super. Pisa Cl. Sci. (5)},
      volume={6},
      number={2},
       pages={195\ndash 261},
      review={\MR{2352517}},
}

\bib{Min10}{article}{
      author={Mingione, G.},
       title={Gradient estimates below the duality exponent},
        date={2010},
        ISSN={0025-5831},
     journal={Math. Ann.},
      volume={346},
      number={3},
       pages={571\ndash 627},
         url={http://dx.doi.org/10.1007/s00208-009-0411-z},
      review={\MR{2578563}},
}

\bib{Min11b}{article}{
      author={Mingione, G.},
       title={Nonlinear measure data problems},
        date={2011},
        ISSN={1424-9286},
     journal={Milan J. Math.},
      volume={79},
      number={2},
       pages={429\ndash 496},
         url={http://dx.doi.org/10.1007/s00032-011-0168-1},
      review={\MR{2862024}},
}

\bib{Ngu15}{article}{
      author={Nguyen, Q.-H.},
       title={Global estimates for quasilinear parabolic equations on
  {R}eifenberg flat domains and its applications to {R}iccati type parabolic
  equations with distributional data},
        date={2015},
        ISSN={0944-2669},
     journal={Calc. Var. Partial Differential Equations},
      volume={54},
      number={4},
       pages={3927\ndash 3948},
         url={http://dx.doi.org/10.1007/s00526-015-0926-y},
      review={\MR{3426099}},
}

\bib{Ngu_pre}{article}{
      author={Nguyen, Q.-H.},
       title={Potential estimates and quasilinear parabolic equations with
  measure data},
     journal={Mem. Amer. Math. Soc.},
       pages={to appear},
}

\bib{NP19}{article}{
      author={Nguyen, Q.-H.},
      author={Phuc, N.~C.},
       title={Good-{$\lambda$} and {M}uckenhoupt-{W}heeden type bounds in
  quasilinear measure datum problems, with applications},
        date={2019},
        ISSN={0025-5831},
     journal={Math. Ann.},
      volume={374},
      number={1-2},
       pages={67\ndash 98},
         url={https://doi.org/10.1007/s00208-018-1744-2},
      review={\MR{3961305}},
}

\bib{NP20}{article}{
      author={Nguyen, Q.-H.},
      author={Phuc, N.~C.},
       title={Pointwise gradient estimates for a class of singular quasilinear
  equations with measure data},
        date={2020},
        ISSN={0022-1236},
     journal={J. Funct. Anal.},
      volume={278},
      number={5},
       pages={108391, 35 pp},
         url={https://doi.org/10.1016/j.jfa.2019.108391},
      review={\MR{4046205}},
}

\bib{NP20b}{article}{
      author={Nguyen, Q.-H.},
      author={Phuc, N.~C.},
       title={Existence and regularity estimates for quasilinear equations with
  measure data: the case $1<p\leq \frac{3n-2}{2n-1}$},
    journal={Anal. PDE},
      pages={to appear},
}


\bib{PS22}{article}{
      author={Park, J.-T.},
      author={Shin, P.},
       title={Regularity estimates for singular parabolic measure data problems
  with sharp growth},
        date={2022},
        ISSN={0022-0396},
     journal={J. Differential Equations},
      volume={316},
       pages={726\ndash 761},
         url={https://doi.org/10.1016/j.jde.2022.01.037},
      review={\MR{4379305}},
}

\bib{Pet08}{article}{
      author={Petitta, F.},
       title={Renormalized solutions of nonlinear parabolic equations with
  general measure data},
        date={2008},
        ISSN={0373-3114},
     journal={Ann. Mat. Pura Appl. (4)},
      volume={187},
      number={4},
       pages={563\ndash 604},
         url={http://dx.doi.org/10.1007/s10231-007-0057-y},
      review={\MR{2413369}},
}

\bib{PPP11}{article}{
      author={Petitta, F.},
      author={Ponce, A.~C.},
      author={Porretta, A.},
       title={Diffuse measures and nonlinear parabolic equations},
        date={2011},
        ISSN={1424-3199},
     journal={J. Evol. Equ.},
      volume={11},
      number={4},
       pages={861\ndash 905},
         url={https://doi.org/10.1007/s00028-011-0115-1},
      review={\MR{2861310}},
}

\bib{PP15}{article}{
      author={Petitta, F.},
      author={Porretta, A.},
       title={On the notion of renormalized solution to nonlinear parabolic
  equations with general measure data},
        date={2015},
        ISSN={2296-9020},
     journal={J. Elliptic Parabol. Equ.},
      volume={1},
       pages={201\ndash 214},
         url={https://doi.org/10.1007/BF03377376},
      review={\MR{3403419}},
}

\bib{Pie83}{article}{
      author={Pierre, M.},
       title={Parabolic capacity and {S}obolev spaces},
        date={1983},
        ISSN={0036-1410},
     journal={SIAM J. Math. Anal.},
      volume={14},
      number={3},
       pages={522\ndash 533},
         url={https://doi.org/10.1137/0514044},
      review={\MR{697527}},
}

\bib{Pri97}{article}{
      author={Prignet, A.},
       title={Existence and uniqueness of ``entropy'' solutions of parabolic
  problems with {$L^1$} data},
        date={1997},
        ISSN={0362-546X},
     journal={Nonlinear Anal.},
      volume={28},
      number={12},
       pages={1943\ndash 1954},
         url={https://doi.org/10.1016/S0362-546X(96)00030-2},
      review={\MR{1436364}},
}

\bib{Urb08}{book}{
      author={Urbano, J.~M.},
       title={The method of intrinsic scaling},
      series={Lecture Notes in Mathematics},
   publisher={Springer-Verlag, Berlin},
        date={2008},
      volume={1930},
        ISBN={978-3-540-75931-7},
         url={https://doi.org/10.1007/978-3-540-75932-4},
        note={A systematic approach to regularity for degenerate and singular
  PDEs},
      review={\MR{2451216}},
}

\bib{Vaz06}{book}{
      author={V\'{a}zquez, J.~L.},
       title={Smoothing and decay estimates for nonlinear diffusion equations},
      series={Oxford Lecture Series in Mathematics and its Applications},
   publisher={Oxford University Press, Oxford},
        date={2006},
      volume={33},
        ISBN={978-0-19-920297-3; 0-19-920297-4},
         url={https://doi.org/10.1093/acprof:oso/9780199202973.001.0001},
        note={Equations of porous medium type},
      review={\MR{2282669}},
}

\end{biblist}
\end{bibdiv}


\end{document}